\documentclass[hidelinks,onefignum,onetabnum]{siamart220329}

\usepackage{lipsum}
\usepackage{amsfonts}
\usepackage{amssymb}
\usepackage{stmaryrd}
\usepackage{graphicx}
\usepackage{epstopdf}
\usepackage{algorithmic}
\usepackage{enumitem}
\usepackage{booktabs}
\usepackage{pgfplotstable}
\usepackage{colortbl}
\usetikzlibrary{pgfplots.groupplots}
\ifpdf
  \DeclareGraphicsExtensions{.eps,.pdf,.png,.jpg}
\else
  \DeclareGraphicsExtensions{.eps}
\fi

\newsiamremark{remark}{Remark}
\newsiamremark{hypothesis}{Hypothesis}
\crefname{hypothesis}{Hypothesis}{Hypotheses}
\newsiamthm{claim}{Claim}

\headers{Unique continuation for the wave equation}{E. Burman and J. Preuss}

\title{Unique continuation for the wave equation based on a discontinuous Galerkin time discretization\thanks{Submitted to the editors DATE.
\funding{This work was funded EPSRC grant EP/V050400/1.}}}

\author{Erik Burman\thanks{Department of Mathematics, University College London, United Kingdom 
  (\email{e.burman@ucl.ac.uk}).}
\and Janosch Preuss\thanks{Department of Mathematics, University College London, United Kingdom 
  (\email{j.preuss@ucl.ac.uk}).}
 }

\usepackage{amsopn}

\newcommand{\bel}{\begin{equation} \label}
\newcommand{\ee}{\end{equation}}
\def\beq{\begin{equation}}
\def\eeq{\end{equation}}
\newcommand{\jump}[1]{\llbracket#1\rrbracket}

\newcommand{\tnorm}[1]{\vert\hspace{-0.3mm}\Vert#1\Vert\hspace{-0.3mm}\vert}
\providecommand{\abs}[1]{\left\lvert#1\right\rvert}
\providecommand{\norm}[1]{\left\lVert#1\right\rVert}
\renewcommand{\leq}{\leqslant}
\renewcommand{\geq}{\geqslant}
\providecommand{\abs}[1]{\left\lvert#1\right\rvert}

\newcommand{\dT}{\mathrm{d}t}
\newcommand{\dX}{\mathrm{d}x}
\newcommand{\dS}{\mathrm{d}S}
\newcommand{\STdom}{Q}
\newcommand{\STdata}{\omega_T}

\newcommand{\SemiDiscrSpace}[1]{ W^{ {#1}}_{h} }
\newcommand{\ProdSemiDiscrSpace}[1]{ \mathcal{W}^{ {#1} }_{h} }
\newcommand{\FullyDiscrSpace}[2]{ W^{ {#1},{#2}}_{h, \Delta t  } }
\newcommand{\FullyDiscrSpaceHat}[2]{ \hat{W}^{ {#1},{#2}}_{h, \Delta t  } }
\newcommand{\ProdFullyDiscrSpace}[2]{ \mathcal{W}^{ {#1},{#2}}_{h, \Delta t  } }
\newcommand{\tangular}[1]{ \llbracket\kern-0.5ex|#1|\kern-0.5ex\rrbracket} 

\newcommand{\drawSquare}{\begin{tikzpicture}
\node[ ] at (0,0) {\textcolor{blue}{\nullfont\pgfuseplotmark{square}}};
\end{tikzpicture} }
\newcommand{\drawCircle}{  
\begin{tikzpicture}
\node[ ] at (0,0) {\textcolor{red}{\nullfont\pgfuseplotmark{o}}};  
\end{tikzpicture} }
\newcommand{\drawTriangle}{\begin{tikzpicture}
\node[ ] at (0,0) {\textcolor{green!70!black}{\nullfont\pgfuseplotmark{triangle}}};
\end{tikzpicture}
}
\pgfplotstableset{
    every head row/.style={before row=\bottomrule,after row=\hline},
    every last row/.style={after row=\toprule},
}

\ifpdf
\hypersetup{
  pdftitle={ },
  pdfauthor={E. Burman, and J. Preuss}
}
\fi

\begin{document}

\maketitle

\begin{abstract}
We consider a stable unique continuation problem for the wave equation where the initial data is lacking and the solution is reconstructed using measurements in some subset of the bulk domain. Typically fairly sophisticated space-time methods have been used in previous work to obtain stable and accurate solutions to this reconstruction problem. Here we propose to solve the problem using a standard discontinuous Galerkin method for the temporal discretization and continuous finite elements for the space discretization. Error estimates are established under a geometric control condition. We also investigate two preconditioning strategies which can be used to solve the arising globally coupled space-time system by means of simple time-stepping procedures. Our numerical experiments test the performance of these strategies and highlight the importance of the geometric control condition for reconstructing the solution beyond the data domain.  
\end{abstract}

\begin{keywords}
Unique continuation, data assimilation, wave equation, finite element method, discontinuous Galerkin, preconditioning, geometric control condition
\end{keywords}

\begin{MSCcodes}
65M32, 35R30, 35L05, 65F08
\end{MSCcodes}

\section{Introduction}
In this article we consider a Lipschitz-stable unique continuation problem for the wave equation. 
For a bounded domain $\Omega \subset \mathbb{R}^d$, where $d \in \mathbb{N}$, with smooth strictly convex boundary $\partial \Omega$, and $T > 0$ we define the space-time 
cylinder $\STdom = (0,T) \times \Omega$ with lateral boundary $\Sigma := (0,T) \times \partial \Omega$.
Analogously, for a subset $\omega \subset \Omega$ we define the data domain $ \STdata := (0,T) \times \omega$.  
We are now concerned with the following data assimilation problem: 
Find $u:\STdom \rightarrow \mathbb{R}$ such that 
\begin{align}
  (\Box u, u)  &= (0,0) \text{ in } \STdom \times \Sigma,    \label{eq:PDE+lateralBc} \\  
  u  &= u_{\omega} \text{ in }  \STdata, \label{eq:data-constraint} 
\end{align} 
where $\Box u :=\partial_t^2 u - \triangle u$ and $u_{\omega}$ is the known data. Note that only the restriction of 
$u$ to the data set is given, whereas the wave-displacement and velocity at time $t=0$ in $\Omega$ are
unknown.

Problem \cref{eq:PDE+lateralBc}-\cref{eq:data-constraint} enjoys Lipschitz stability if the 
set $ \STdata \subset \STdom$ fulfills the so called geometric control condition (GCC), see \cite{BLR88,BLR92} for a precise definition thereof.
\begin{theorem}[Lipschitz stability]\label{thm:Lipschitz-stab}
Assume that $ \STdata \subset \STdom$ fulfills the GCC.
Then there exists a constant $C > 0$ such that for any $\phi \in H^1(\STdom)$ the following stability estimate holds: 
\bel{eq:Lipschitz-stability}
\norm{\phi}_{L^{\infty}(0,T;L^2(\Omega))} + \norm{\partial_t \phi}_{L^{2}(0,T;H^{-1}(\Omega))} 
	\leq C \left(  \norm{\phi}_{ L^2(\STdata) } + \norm{\phi}_{ L^2(\Sigma)} + \norm{ \Box \phi}_{ H^{-1} ( \STdom ) }     \right).
\ee 
\end{theorem}
\begin{proof}
See \cite[Remark A.5]{BFMO21control}. 
\end{proof} 
Unless otherwise stated, below we assume that the GGC holds for the chosen geometry.

\subsection{Literature review}\label{ssection:lit-review}
Recently, at least five different methods \cite{CM15,BFO20,BFMO21,MM21,DMS23} for the numerical solution of problem \cref{eq:PDE+lateralBc}-\cref{eq:data-constraint} have been proposed. 
Before discussing these approaches let us briefly explain a common principle that drives the design of all numerical methods based on the stability estimate \cref{eq:Lipschitz-stability}. To this end, let $u_h \in H^1(\STdom)$ be some discrete approximation of the solution $u$ of problem \cref{eq:PDE+lateralBc}-\cref{eq:data-constraint}.
The idea is then to apply \Cref{thm:Lipschitz-stab} to $\phi = u - u_h$, which leads to an error estimate provided that the terms on the right hand side of \cref{eq:Lipschitz-stability}, in particular the PDE residual $\norm{ \Box (u - u_h ) }_{ H^{-1} ( \STdom ) }$, can be controlled. Let us distinguish the recently proposed methods based on how they strive to 
achieve this aim.
\begin{itemize}
\item The first method proposed in \cite{CM15} in fact preceeded the stability estimate in \cref{eq:Lipschitz-stability}. Here the discretization was required to fulfill 
$\Box u_h \in L^2(0,T,H^{-1}(\Omega))$ to repect the PDE constraint. To this end, $C^1$-conforming (space-time) elements have been used. Verifying discrete inf-sup stability in 
this setting is an intricate problem. Alternatively, the authors propose a stabilized formulation which is shown to convergence under an additional regularity assumption on 
an auxiliary variable, see \cite[Proposition 2]{CM15}.
\item Subsequently, a class of methods \cite{BFO20,BFMO21} based on a primal-dual stabilized finite element framework introduced in \cite{B13} were proposed. Here, the problem is formulated as a PDE constrained optimization problem and discretized. Then, (weakly)-consistent stabilization terms are added to the discretization which allow to control the PDE residual without having to evaluate dual norms or to use exotic discretization spaces.  
Based on this strategy, in reference \cite{BFO20} a piecewise affine continuous finite element method in space was coupled with backward differences in time to yield a provably convergent 
first order method for problem \cref{eq:PDE+lateralBc}-\cref{eq:data-constraint}. 
We point out that even though backward differences in time were used, this is not a time-stepping method, as the solution of a globally coupled space-time system is required. \par 
Since the method of \cite{BFO20} ultimately leads to space-time systems anyhow, it seems natural to treat problem \cref{eq:PDE+lateralBc}-\cref{eq:data-constraint}  
directly using a full space-time discretization. In fact, such an approach has been considered in reference \cite{BFMO21}. Therein the authors proposed an arbitrary order $H^1(\STdom)$-conforming
stabilized finite element method and derived optimal error estimates. An extension of the method to a closely related control problem has been presented in \cite{BFMO21control}. 
\item Whereas the methods introduced in \cite{BFO20,BFMO21} avoid a direct evaluation of the PDE residual $\norm{ \Box (u - u_h ) }_{ H^{-1} ( \STdom ) }$, in reference \cite{DMS23} a least-squares finite element method has been introduced which minimizes this quantity directly. The authors deal with the arising challenges of having to prove technical inf-sup conditions and the need to replace the negative index Sobolev norms by computable quantities to be able to solve the problem numerically. We remark that the framework introduced in this reference is fairly general which allows the additional treatment of data assimilation for the heat equation and the Cauchy problem for Poisson's equation in this publication.
\item Finally, physics informed neural networks have been proposed in reference \cite{MM21} to solve problem \cref{eq:PDE+lateralBc}-\cref{eq:data-constraint} and similar data assimilation problems. Here, the PDE residual is minimized in the $L^2$ instead of the $H^{-1}$-norm on a set of training points and  estimates on the generalization error of these networks have been derived based on \Cref{thm:Lipschitz-stab}. Let us mention that, even though this approach in its present form can yield accurate results for problems enjoying robust stability, it has recently been shown that it is not suitable yet to solve Helmholtz problems posed in geometrical configurations with poor stability, see reference \cite{N23}.  
\end{itemize}

\subsection{Motivation for and novelty of this paper} 
The methods discussed in \cref{ssection:lit-review} share the common feature of treating problem \cref{eq:PDE+lateralBc}-\cref{eq:data-constraint} in a space-time fashion.
This is reasonable since initial data $u(0,\cdot), u_t(0,\cdot)$ for starting a time-marching procedure is not available. 
Also, in order to apply \Cref{thm:Lipschitz-stab} to the error, one needs to ensure that $u_h$ is in $H^1(\STdom)$, which renders the use of conforming space-time discretizations very natural and appealing.
However, space-time methods also come with some considerable challenges. Applying for example the method from \cite{BFMO21} to a domain $\Omega \subset \mathbb{R}^3$ would 
require to generate a non-trivial (recall that $\partial \Omega$ is smooth) four-dimensional space-time mesh and a corresponding discretization which can be challenging to realize with common finite element software. Whereas this algorithmic challenge may be overcome, solving the arising extremely large and ill-conditioned linear systems appears like 
a more daunting obstacle. Therefore, we want to investigate in this article to what extent problem \cref{eq:PDE+lateralBc}-\cref{eq:data-constraint} is amenable to the solution with computationally less-demanding time-marching methods. \par
To this end, we consider a fairly flexible time discretization provided by a discontinuous Galerkin method (dG), which is very well-established in the literature, see e.g. \cite{HH88,J93,F93}. The dG method is known to allow for time-marching procedures provided that only weak continuity in time is imposed by the variational formulation. The analysis presented in this article shows however, that some stronger constraints which couple the problem globally in time cannot be avoided if one wants to preserve the optimal error estimates from reference \cite{BFMO21}. The contribution of this manuscript is to identify these couplings explictly and then to introduce small perturbations thereof for which time-marching is possible.
These approximate solvers can then be used as preconditioners to solve the original globally-coupled space-time problem iteratively. 
The advantage of this strategy is that it allows to solve problem \cref{eq:PDE+lateralBc}-\cref{eq:data-constraint} using a standard time-marching procedure and that it avoids space-time discretizations. Now we can compute solutions in spatially-three dimensional domains ($d=3$) up to moderate levels of precision. The current limitations of our approach are that the preconditioners lack robustness with respect refinement of the discretization and that adaptive refinement strategies have not been considered yet. 
\par 
We remark that a dG method in time for an ill-posed heat equation has already been analyzed in reference \cite{BDE23}. However, the authors did not make use of the dG in time structure for solving the linear systems. This idea appears novel to us.

\subsection{Structure of this paper}
To show the generality of the approach with respect to time discretisation we consider both a semi-discretisation in space, with time continuous and a fully-discrete formulation with a discontinuous Galerkin approximation in time.
We introduce these methods in \cref{section:discr} and establish some auxiliary results valid for both approaches. The semi-discrete method in which time is still considered to be a continuous variable is analyzed in \cref{section:error-analysis-semi-discrete}, whereas analysis of the fully-discrete method can be found in \cref{section:error-analysis-fully-discrete}. The separate treatment of these two approaches allows us to identify very clearly the mechanism which couples provably convergent approximations of \cref{eq:PDE+lateralBc}-\cref{eq:data-constraint} globally in time at the discrete level. 
Based on these insights, we propose two preconditioning strategies in \cref{section:precond} which can be used for the iterative solution of the arising linear systems. In \cref{section:numexp} we investigate their performance in numerical experiments and explore further aspects which go beyond our analytical results. For instance, we investigate how the numerical method behaves when the GCC (which is crucial for the analysis) is violated. This article is accompanied by a \cref{appendix:space-time-interp} in which interpolation into tensor-product space-time finite element spaces is discussed.

\section{A semi-discrete and a fully-discrete method for unique continuation}\label{section:discr}
In this section we introduce two levels of discretization for problem \cref{eq:PDE+lateralBc}-\cref{eq:data-constraint}.
We start in \cref{ssection:discr-geom} by defining a geometric partition of the the space-time cylinder $Q$ 
into time-slabs, which are tensor products of subintervals of $(0,T)$ with a simplicial mesh covering the spatial 
domain $\Omega$. In \cref{ssection:semi-discrete-method} we then introduce a preliminary method to solve problem \cref{eq:PDE+lateralBc}-\cref{eq:data-constraint}, which 
features an ansatz space that is discrete in the spatial variable yet still continuous in time. 
In \cref{ssection:fully-discrete-method} we drop the latter assumption by introducing a
discontinuous Galerkin time discretization. This requires a slight enrichment of the method introduced in \cref{ssection:semi-discrete-method} to impose sufficient temporal continuity 
by means of the variational formulation. In \cref{ssection:stab-interp} and  \cref{ssection:consistency-res} we establish some theoretical results valid for both discretizations which are required for the error analysis of the respective methods presented in \cref{section:error-analysis-semi-discrete} and \cref{section:error-analysis-fully-discrete}.

\subsection{Discretization of the space-time cylinder}\label{ssection:discr-geom}
\subsubsection{Spatial mesh}
For \cref{thm:Lipschitz-stab} to hold we require that $\Omega$ has a smooth boundary. To account for this in a simple manner in the spatial disretization we use the same strategy as discussed in \cite[Section 4.2]{BFMO21control}. Essentially, we fit the boundary using curved elements and incorporate the boundary conditions weakly using Nitsche's method, similar as in \cite[Theorem 2.1]{Thom07}. \par 
Let $\hat{\mathcal{T}}_h$ be a quasi-uniform triangulation such that $\hat{\Omega}:= \cup_{K \in \hat{\mathcal{T}}_h} K $ contains $\Omega$, i.e.\ $\Omega \subset \hat{\Omega}$. We then set $
\mathcal{T}_h := \{ K \cap \Omega \mid K \in  \hat{\mathcal{T}}_h \}.$
Note that by construction $\Omega = \cup_{K \in \mathcal{T}_h} K$ holds true.
As explained in \cite[Section 4.2]{BFMO21control}, it is possible to construct the family $\{ \hat{\mathcal{T}}_h \mid h > 0 \}$ such that for sufficiently small $h$ the  continuous trace inequality
\bel{ieq:cont-trace-ieq}
\norm{v}_{ [L^2(\partial K)]^d } \leq C \left( h^{-1/2} \norm{v}_{ [L^2(K)]^d } + h^{1/2} \norm{ \nabla v  }_{ [L^2(K)]^d }   \right), \; \forall v \in [H^1(K)]^d, \; K \in \mathcal{T}_h
\ee
holds, where here as in the remainder of this article $C$ denotes a generic constant independent of $h$. We refer to \cite[Appendix B]{BFMO21control} for a proof. 
We will in the following also assume that the triangulation $\mathcal{T}_h$ fits the data domain $\omega$. 
Denote then by $\mathcal{F}_i$ the set of all interior facets of $\mathcal{T}_h$ and let 
\[
V_h^k := \{ u \in H^1(\Omega) \mid u|_K \in \mathcal{P}_k(K) \; \forall K \in \mathcal{T}_h \},
\]
where $\mathcal{P}_k(K)$ denotes the spaces of polynomials of order at most $k \in \mathbb{N}$ on $K$. 
\subsubsection{Partition of the time axis} 
We introduce a partition $0 = t_0 < t_1 < \ldots < t_N=T$ of the time axis into subintervals $I_n := (t_n,t_{n+1}), n=0,\ldots,N-1$.
For simplicity all time intervals are assumed to be of equal length $\Delta t = \abs{t_{n+1} - t_n}$.
The discretization of the time axis implies a corresponding partition of the space-time cylinder into time slabs
\bel{eq:def-time-slab}
Q^n := I_n \times \Omega, \; \Sigma^n = I_n \times \Sigma, \; n=0,\ldots,N-1, \quad Q := \cup_{n=0}^{N-1} Q^n, \; \Sigma := \cup_{n=0}^{N-1} \Sigma^n.
\ee
The space-time integrals on the slabs are denoted by
\[ 
(u,v)_{Q^n} := \int\limits_{I_n} \int\limits_{\Omega} u v \; \dX \; \dT, 
\;
a(u,v)_{Q^n} := \! \int\limits_{I_n} \int\limits_{\Omega} \nabla u \cdot \nabla v \; \dX \; \dT,
\;
(u,v)_{ \Sigma^n } = \int\limits_{I_n}  \int\limits_{ \partial \Omega} u v \; \dS \; \dT 
\]
and we set $\norm{v}^2_{Q^n} := (v,v)_{Q^n}$ and $ \norm{v}^2_{\Sigma^n} = (v,v)_{\Sigma^n}$. 
The data domain and integrals thereon are defined similarly
\[
\omega^n := I_n \times \omega, \quad (u,v)_{\omega^n} := \int\limits_{I_n} \int\limits_{\omega} u v \; \dX \; \dT,   \quad \norm{ u }_{ \STdata }^2 = \sum\limits_{n=0}^{N-1}(u,v)_{\omega^n}.
\]

\subsection{A semi-discrete conforming FEM}\label{ssection:semi-discrete-method}
Let us first ignore the partition of the time axis in our ansatz space and define
\bel{eq:def-SemiDiscSpace}
 \SemiDiscrSpace{k} := H^1 \left( 0,T; V_h^k \right).
\ee 
For tuples in the product space $ \ProdSemiDiscrSpace{k} = \SemiDiscrSpace{k} \times \SemiDiscrSpace{k} $ it is convenient to use the notation    
\bel{eq:mixed-vars} 
  \underline{\mathbf{U_h}} := ( \underline{u}_1,\underline{u}_2) \in \ProdSemiDiscrSpace{k}.
 \ee
Let us now proceed to define a variational method for solving problem \cref{eq:PDE+lateralBc}-\cref{eq:data-constraint} in these semi-discrete spaces. It is convenient to define this method already in a time-slab wise manner, since we will later use a slightly augmented version thereof once the time variable has been discretized by a discontinuous Galerkin method.
Define the bilinear form representing a mixed formulation of the wave equation as 
\begin{equation}\label{eq:bil-form-def-fully-discrete}
A[ \underline{\mathbf{U}}_h, \underline{\mathbf{Y}}_h] = \! \sum\limits_{n=0}^{N-1} \! \left\{ (\partial_t \underline{u}_2, \underline{y}_1)_{Q^n} + a(\underline{u}_1,\underline{y}_1)_{Q^n}  
    + (\partial_t \underline{u}_1 - \underline{u}_2, \underline{y}_2)_{Q^n}    
  - ( \nabla \underline{u}_1 \! \cdot \! \mathbf{n}, \underline{y}_1)_{ \Sigma^n } 
  \right\}
\end{equation}
for $\underline{\mathbf{U_h}} \in \ProdSemiDiscrSpace{k}$ and $\underline{\mathbf{Y}}_h \in $ for $\ProdSemiDiscrSpace{k_{\ast}}$ with $k_{\ast} \in \mathbb{N}$. Here $\mathbf{n}$ denotes the outer normal vector of $\Omega$.  
Note that for a sufficiently regular solution $u$ of \cref{eq:PDE+lateralBc} we obtain for $ \mathbf{U} = (u, \partial_t u)$ via integration by parts in space that $A[ \mathbf{U} , \underline{\mathbf{Y}}_h] = 0$.
We will search for approximate solutions of problem \cref{eq:PDE+lateralBc}-\cref{eq:data-constraint} as stationary points of the Lagrangian
\bel{eq:Lagrangian}
\mathcal{L}_h( \underline{\mathbf{U}}_h,\underline{\mathbf{Z}}_h) :=
\frac{1}{2} \norm{ \underline{u}_1 - u_{\omega}}_{\omega_T}^2 +  A[ \underline{\mathbf{U}}_h, \underline{\mathbf{Z}}_h ] + \frac{1}{2} S_h( \underline{\mathbf{U}}_h, \underline{\mathbf{U}}_h)  - \frac{1}{2} S_h^{\ast}( \underline{\mathbf{Z}}_h, \underline{\mathbf{Z}}_h), 
\quad 
\ee
for $(\underline{\mathbf{U}}_h,\underline{\mathbf{Z}}_h) \in \ProdSemiDiscrSpace{k} \times \ProdSemiDiscrSpace{k_{\ast}}$.
Here, the first and second term of the Lagrangian incorporate the data and PDE constraints, respectively. The contributions $ S_h(\cdot,\cdot)$ and $S_h^{\ast}(\cdot,\cdot)$ represent stabilization terms which are defined as follows. We define a spatial continuous interior penalty term $J(\cdot,\cdot)$, which penalizes the jump of the gradient $\jump{ \nabla \underline{u}_1 }$ over interior facets, a Galerkin-least squares term $G(\cdot,\cdot)$ which enforces the PDE at the element level, a term $I_{0}(\cdot,\cdot)$ which ensures that $\underline{u}_2 = \partial_t \underline{u}_1$ and a term for boundary stability:
\begin{align*}
& J( \underline{\mathbf{U}}_h, \underline{\mathbf{W}}_h ) := \sum\limits_{n=0}^{N-1} \int\limits_{I_n} \sum\limits_{F \in \mathcal{F}_i} h  ( \jump{ \nabla \underline{u}_1 } , \jump{ \nabla \underline{w}_1 }  )_{F} \; \dT,  
\\ 
& G( \underline{\mathbf{U}}_h, \underline{\mathbf{W}}_h ) := \sum\limits_{n=0}^{N-1} \int\limits_{I_n} \sum\limits_{K \in \mathcal{T}_h }  h^2 ( \partial_t \underline{u}_2 - \Delta \underline{u}_1 , \partial_t \underline{w}_2 - \Delta \underline{w}_1 )_{K}  \; \dT, \quad   
\\
 & I_0( \underline{\mathbf{U}}_h, \underline{\mathbf{W}}_h ) := \sum\limits_{n=0}^{N-1}(\underline{u}_2 - \partial_t \underline{u}_1, \underline{w}_2 - \partial_t \underline{w}_1)_{Q^n}, 
\quad  R( \underline{\mathbf{U}}_h, \underline{\mathbf{W}}_h ) := \sum\limits_{n=0}^{N-1} h^{-1} ( \underline{u}_1, \underline{w}_1)_{ \Sigma^n }.
\end{align*}
The complete primal stabilization is given as a sum of these terms:
\begin{equation}\label{eq:Def-Sh}
S_{h}( \underline{\mathbf{U}}_h, \underline{\mathbf{W}}_h ) := 
J( \underline{\mathbf{U}}_h, \underline{\mathbf{W}}_h ) 
+ G( \underline{\mathbf{U}}_h, \underline{\mathbf{W}}_h ) 
+ R( \underline{\mathbf{U}}_h, \underline{\mathbf{W}}_h )
+ I_0( \underline{\mathbf{U}}_h, \underline{\mathbf{W}}_h ).
\end{equation}
To stabilize the dual variable $ \underline{\mathbf{Z}}_h$ we will use:
\begin{equation}\label{eq:dual_stab_fully_disc}
S^{\ast}_h( \underline{\mathbf{Y}}_h, \underline{\mathbf{Z}}_h) := 
	 \sum\limits_{n=0}^{N-1} \big\{ (\underline{y}_1,\underline{z}_1)_{Q^n} + a(\underline{y}_1,\underline{z}_1)_{Q^n} + (\underline{y}_2,\underline{z}_2)_{Q^n} + h^{-1} (\underline{y}_1 ,\underline{z}_1 )_{ \Sigma^n }  \big\}. 
\end{equation}

These stabilizations give rise to semi-norms, respectively norms:
\begin{equation}\label{eq:def-h-semi-norms}
\abs{ \underline{\mathbf{U}}_h }_{S_h} := S_h( \underline{\mathbf{U}}_h, \underline{\mathbf{U}}_h)^{1/2}, \quad
\norm{ \underline{\mathbf{Z}}_h }_{S_h^{\ast}} := S_h^{\ast}( \underline{\mathbf{Z}}_h, \underline{\mathbf{Z}}_h)^{1/2}. 
\end{equation}
We will later see that the following expression defines a norm on $\ProdSemiDiscrSpace{k} \times \ProdSemiDiscrSpace{k_{\ast}}$: 
\bel{eq:triple_norm_h_def}
\tnorm{ ( \underline{\mathbf{U}}_h, \underline{\mathbf{Z}}_h) }^2_h  :=  \abs{ \underline{\mathbf{U}}_h }_{S_h}^2 +  \norm{ \underline{u}_1 }^2_{\STdata} +  \norm{ \underline{\mathbf{Z}}_h }_{S_h^{\ast}}^2. 
\ee
The first order optimality conditions characterizing critical points of the Lagrangian take the form: 
Find $ (\underline{\mathbf{U}}_h ,\underline{\mathbf{Z}}_h) \in  \ProdSemiDiscrSpace{k} \times \ProdSemiDiscrSpace{k_{\ast}}  $ such that
\begin{align}\label{eq:optimality_cond_semi}
	(\underline{u}_1,\underline{w}_1)_{\STdata} + A[ \underline{\mathbf{W}}_h, \underline{\mathbf{Z}}_h ] +  S_h(  \underline{\mathbf{U}}_h, \underline{\mathbf{W}}_h ) &= (u_{\omega},\underline{w}_1)_{ \STdata}, \quad \forall \; \underline{\mathbf{W}}_h \in  \ProdSemiDiscrSpace{k}, \nonumber \\ 
A[ \underline{\mathbf{U}}_h, \underline{\mathbf{Y}}_h ] - S_h^{\ast}( \underline{ \mathbf{Y}}_h, \underline{\mathbf{Z}}_h)  &= 0, \quad \forall \;  \underline{\mathbf{Y}}_h \in \ProdSemiDiscrSpace{k_{\ast}}. 
\end{align} 
This can be written in a more compact way: 
Find $ (\underline{\mathbf{U}}_h ,\underline{\mathbf{Z}}_h) \in  \ProdSemiDiscrSpace{k} \times \ProdSemiDiscrSpace{k_{\ast}} $ such that
\begin{equation}\label{eq:opt_compact_semi}
B_h[ ( \underline{\mathbf{U}}_h,\underline{\mathbf{Z}}_h), ( \underline{\mathbf{W}}_h ,\underline{\mathbf{Y}}_h ) ] = (u_{\omega_T}, \underline{w}_1)_{ \STdata},
\end{equation}
where  
\begin{align}\label{eq:complete_bfi_def_semi}
B_h[ ( \underline{\mathbf{U}}_h,\underline{\mathbf{Z}}_h), ( \underline{\mathbf{W}}_h ,\underline{\mathbf{Y}}_h ) ]  := & (\underline{u}_1,\underline{w}_1)_{\STdata} + A[\underline{\mathbf{W}}_h, \underline{\mathbf{Z}}_h ] + S_h( \underline{\mathbf{U}}_h, \underline{\mathbf{W}}_h ) \nonumber \\
	& + A[\underline{\mathbf{U}}_h, \underline{\mathbf{Y}}_h ] - S_h^{\ast}( \underline{\mathbf{Y}}_h, \underline{\mathbf{Z}}_h). 
\end{align}
In \cref{section:error-analysis-semi-discrete} we will show 
that the first component $\underline{u}_1$ of the solution of \cref{eq:complete_bfi_def_semi} converges to the exact solution $u$ of \cref{eq:PDE+lateralBc}-\cref{eq:data-constraint} 
in the norm defined by the left hand side of \cref{eq:Lipschitz-stability} at a rate of $h^k$. Note, however, that the method \cref{eq:complete_bfi_def_semi} is not yet suitable for implementation since the time variable is still assumed to be continuous.

\subsection{A fully-discrete FEM using dG in time}\label{ssection:fully-discrete-method}
To discretize time, let $\mathcal{P}^q(I_n)$ denote the space of polynomials up to degree $q \in \mathbb{N}_0$ on $I_n$ and define a discontinuous (in-time) finite element space: 
\[ \FullyDiscrSpace{k}{q} :=  \otimes_{n=0}^{N-1} \mathcal{P}^q(I_n) \otimes V_h^k,
\quad \ProdFullyDiscrSpace{k}{q} := \FullyDiscrSpace{k}{q} \times \FullyDiscrSpace{k}{q},  
\quad q \in \mathbb{N}_{0}, \; k \in \mathbb{N}. \]
Note that functions in $ \ProdFullyDiscrSpace{k}{q} $ are allowed to be discontinuous between time-slabs.
Hence, we require a notation for jumps in time. 
Let
\bel{eq:def-jump-in-time}
v^{n}_{\pm}(x) := \lim_{s \rightarrow 0^{+}} v(x,t_n \pm s), \text{ and } \jump{ v^n } := v_{+}^n - v_{-}^n.
\ee 
Since continuity in time is now no longer ensured by the spaces, we will need to augment the semi-discrete Lagrangian $\mathcal{L}_h(\cdot,\cdot)$ from \cref{eq:Lagrangian} by a stabilization term $S^{\uparrow \downarrow}_{\Delta t}$ which imposes sufficient temporal regularity of the primal variable, i.e.\ 
\bel{eq:Lagrangian_fully_discrete}
\mathcal{L}( \underline{\mathbf{U}}_h,\underline{\mathbf{Z}}_h) :=
\mathcal{L}_h( \underline{\mathbf{U}}_h,\underline{\mathbf{Z}}_h) +  
\frac{1}{2} S^{\uparrow \downarrow}_{\Delta t}( \underline{\mathbf{U}}_h, \underline{\mathbf{U}}_h ), 
\qquad 
(\underline{\mathbf{U}}_h,\underline{\mathbf{Z}}_h) \in \ProdFullyDiscrSpace{k}{q} \times \ProdFullyDiscrSpace{k_{\ast}}{q_{\ast}},
\ee
where
\begin{equation}\label{eq:def-jump-stab}
S^{\uparrow \downarrow}_{\Delta t}( \underline{\mathbf{U}}_h, \underline{\mathbf{W}}_h ) 
:= \underline{I}_1^{\uparrow \downarrow}( \underline{\mathbf{U}}_h, \underline{\mathbf{W}}_h ) + 
\underline{I}_2^{\uparrow \downarrow}( \underline{\mathbf{U}}_h, \underline{\mathbf{W}}_h ) 
\end{equation} 
\begin{align*}
\underline{I}_1^{\uparrow \downarrow}( \underline{\mathbf{U}}_h, \underline{\mathbf{W}}_h ) &:= \sum\limits_{n=1}^{N-1} \left\{ \frac{1}{\Delta t} ( \jump{ \underline{u}_1^n } , \jump{ \underline{w}_1^n } )_{ \Omega} 
	+ \Delta t ( \jump{ \nabla \underline{u}_1^n } , \jump{ \nabla \underline{w}_1^n } )_{ \Omega} \right\}, \\
\underline{I}_2^{\uparrow \downarrow}( \underline{\mathbf{U}}_h, \underline{\mathbf{W}}_h ) &:= \sum\limits_{n=1}^{N-1} \frac{1}{\Delta t} ( \jump{ \underline{u}_2^n } , \jump{ \underline{w}_2^n } )_{ \Omega}.
\end{align*}
The first order optimality conditions then take the form:
Find $(\underline{\mathbf{U}}_h, \underline{\mathbf{Z}}_h) \in \ProdFullyDiscrSpace{k}{q} \times \ProdFullyDiscrSpace{ k_{\ast} }{ q_{\ast} } $ such that
\bel{eq:opt_compact_fully_disc} 
B[ ( \underline{\mathbf{U}}_h,\underline{\mathbf{Z}}_h), ( \underline{\mathbf{W}}_h ,\underline{\mathbf{Y}}_h ) ] = (u_{\omega}, \underline{w}_1)_{\STdata }
\ee 
for all $( \underline{\mathbf{W}}_h ,\underline{\mathbf{Y}}_h ) \in \ProdFullyDiscrSpace{k}{q} \times \ProdFullyDiscrSpace{ k_{\ast} }{ q_{\ast} } $, 
where 
\begin{equation}\label{eq:complete_bfi_def_fully_discr}
B[ ( \underline{\mathbf{U}}_h,\underline{\mathbf{Z}}_h), ( \underline{\mathbf{W}}_h ,\underline{\mathbf{Y}}_h ) ] =  B_h[ ( \underline{\mathbf{U}}_h,\underline{\mathbf{Z}}_h), ( \underline{\mathbf{W}}_h ,\underline{\mathbf{Y}}_h ) ] 
+ S^{\uparrow \downarrow}_{\Delta t}( \underline{\mathbf{U}}_h, \underline{\mathbf{W}}_h )  
\end{equation}
with $B_h$ as defined in \cref{eq:opt_compact_semi} and we assume that $k,k_{\ast},q \geq 1$ and $q_{\ast} \geq 0$.   \par 
Since $\tnorm{(\cdot,\cdot)}_h$ no longer defines a norm on the fully discrete spaces $\ProdFullyDiscrSpace{k}{q} \times \ProdFullyDiscrSpace{ k_{\ast} }{ q_{\ast} } $, it has to be augmeneted as well. To this end, we define 
\[
\tnorm{ (\underline{\mathbf{U}}_h, \underline{\mathbf{Z}}_h) }^2  :=  \tnorm{ (\underline{\mathbf{U}}_h, \underline{\mathbf{Z}}_h) }^2_h +\abs{ \underline{\mathbf{U}}_h   }_{ \uparrow \downarrow }^2, \qquad 
\abs{ \underline{\mathbf{U}}_h   }_{ \uparrow \downarrow } := S^{\uparrow \downarrow}_{\Delta t}( \underline{\mathbf{U}}_h, \underline{\mathbf{U}}_h )^{1/2}.
\]

\subsection{Stability, continuity and interpolation results}\label{ssection:stab-interp}
In this subsection we derive some important auxiliary results which are later used in the error analysis of the semi-discrete and the fully-discrete method. The first aim is to establish inf-sup stability of the bilinear forms $B_h$ and $B$ on the respective spaces. First we have to verify that the expression $\tnorm{(\cdot,\cdot)}_h$, respectively $\tnorm{(\cdot,\cdot)}$, are indeed norms. This is established in the following lemma and its corollary.
\begin{lemma}\label{lem:IP-norm}
\begin{enumerate}[label=(\alph*)]
\item For $\underline{\mathbf{U}}_h \in \ProdSemiDiscrSpace{k} + \ProdFullyDiscrSpace{k}{q}$
and $y \in H^1_0(\STdom)$ it holds that:
\begin{align}
& \int\limits_{\STdom} \left\{ -(\partial_t \underline{u}_1)  \partial_t y + \nabla \underline{u}_1 \nabla y \right\} 
 =  \sum\limits_{n=0}^{N-1} \int\limits_{I_n} \sum\limits_{K \in \mathcal{T}_h }  ( \partial_t \underline{u}_2 - \Delta \underline{u}_1 , y )_{K} \; \dT \label{eq:IP-H-1} \\ 
& + \sum\limits_{n=0}^{N-1} ( \underline{u}_2 - \partial_t u_1,  \partial_t y )_{Q^n}  + \sum\limits_{n=0}^{N-1} \int\limits_{I_n} \sum\limits_{F \in \mathcal{F}_i}  ( \jump{ \nabla \underline{u}_1 \cdot \mathbf{n} } , y  )_{F} \; \dT 
	+ \sum\limits_{n=1}^{N-1}(  \jump{ \underline{u}_2^n } , y )_{ \Omega}, \nonumber \end{align}
 where $ \jump{ \nabla \underline{u}_1 \cdot \mathbf{n} } $ denotes the jump of the normal derivative over interior facets.
\item If for $(\underline{\mathbf{U}}_h,\underline{\mathbf{Z}}_h) \in (\ProdSemiDiscrSpace{k} + \ProdFullyDiscrSpace{k}{q}) \times  (\ProdSemiDiscrSpace{k_{\ast} } + \ProdFullyDiscrSpace{ k_{\ast} }{ q_{\ast} }) $ it holds $\tnorm{ (\mathbf{U}_h, \mathbf{Z}_h) }_h = 0$, then  
\begin{align}
& \underline{\mathbf{Z}}_h = 0, \quad \underline{u}_1|_{\Sigma} = 0, \quad \underline{u}_1|_{ \STdata } = 0, \quad \partial_t \underline{u}_1 = \underline{u}_2 \label{eq:h-norm-zero-impl} \\& \text{ and } \norm{\Box \underline{u}_1}_{ H^{-1}(\STdom)} = \sup_{\substack{  y \in H^1_0(\STdom), \\ \norm{y}_{H^1(\STdom) } = 1  }} \sum\limits_{n=1}^{N-1}(  \jump{ \underline{u}_2^n } , y )_{ \Omega}.
  \label{eq:h-norm-zero-res-H-1} 
\end{align}
\end{enumerate}
\end{lemma}
\begin{proof}
\begin{enumerate}[label=(\alph*)]
\item This follows immediately from integration by parts. 
\item In view of definitions \cref{eq:Def-Sh}-\cref{eq:triple_norm_h_def}, the claims in \cref{eq:h-norm-zero-impl} follow directly from $\tnorm{ (\underline{\mathbf{U}}_h, \underline{\mathbf{Z}}_h) }_h = 0$. The claim in \cref{eq:h-norm-zero-res-H-1} is a consequence of \cref{eq:IP-H-1}, since by definition 
\[
\norm{\Box \underline{u}_1}_{ H^{-1}(\STdom)} = \sup_{\substack{  y \in H^1_0(\STdom), \\ \norm{y}_{H^1(\STdom) } = 1  }}  \int\limits_{\STdom} \left\{ -(\partial_t \underline{u}_1)  \partial_t y + \nabla \underline{u}_1 \nabla y \right\} 
\]
and all terms on the right hand side of \cref{eq:IP-H-1}, except for the jump over the time slice boundaries, vanish due to $\abs{ \underline{\mathbf{U}}_h }_{S_h} = 0$.
\end{enumerate}
\end{proof}
\begin{corollary}\label{cor:norm}
\begin{enumerate}[label=(\alph*)]
\item It holds that $\tnorm{ (\cdot, \cdot) }_h$ is a norm on  $ \ProdSemiDiscrSpace{k} \times \ProdSemiDiscrSpace{k_{\ast}} $. 
\item The stronger expression $\tnorm{ (\cdot, \cdot) }$ is a norm on  $\ProdFullyDiscrSpace{k}{q} \times \ProdFullyDiscrSpace{ k_{\ast} }{ q_{\ast} }$. 
\end{enumerate}
\end{corollary}
\begin{proof}
\begin{enumerate}[label=(\alph*)]
	\item Let  $\tnorm{ (\underline{\mathbf{U}}_h, \underline{\mathbf{Z}}_h) }_h = 0$. Since $\underline{u}_2 \in H^1(\STdom)$, we have $\Box \underline{u}_1 = 0 $ in $H^{-1}(\STdom)$ from \cref{eq:h-norm-zero-res-H-1}. As $\underline{u}_1 \in  H^1(\STdom)$, we are allowed to apply the stability estimate from \cref{eq:Lipschitz-stability} to $\phi = \underline{u}_1$. Thanks to \cref{eq:h-norm-zero-impl} we then obtain as required that 
\[ \underline{u}_1 = 0, \quad \underline{u}_2 = \partial_t \underline{u}_1 = 0, \quad \text{and }  \underline{\mathbf{Z}}_h = 0.
\]
\item Since $ \abs{\underline{\mathbf{U}}_h }_{ \uparrow \downarrow } = 0$ implies $(\underline{u}_1, \underline{u}_2) \in [H^1(\STdom)]^2$, the same argument as in (a) applies. 
\end{enumerate}
\end{proof}
Due to
\[
	B_h[ (\underline{\mathbf{U}}_h,\underline{\mathbf{Z}}_h), ( \underline{\mathbf{U}}_h ,-\underline{\mathbf{Z}}_h ) ]  = \tnorm{ (\underline{\mathbf{U}}_h, \underline{\mathbf{Z}}_h) }_h \tnorm{ (\underline{\mathbf{U}}_h, -\underline{\mathbf{Z}}_h) }_h, 
\quad 
S^{\uparrow \downarrow}_{\Delta t}( \underline{\mathbf{U}}_h, \underline{\mathbf{U}}_h ) = \abs{ \underline{\mathbf{U}}_h   }_{ \uparrow \downarrow }^2  
\]
and \cref{cor:norm} we then obtain inf-sup stability on the respective spaces:
\begin{align}
 \label{eq:inf-sup}
 \sup_{ (\underline{\mathbf{W}}_h,\underline{\mathbf{Y}}_h) \in  \ProdSemiDiscrSpace{k} \times \ProdSemiDiscrSpace{ k_{\ast} }    } \!\!\!\!\!\!\!\!\!\!\!\! \frac{ B_h[ (\underline{\mathbf{U}}_h,\underline{\mathbf{Z}}_h), ( \underline{\mathbf{W}}_h , \underline{\mathbf{Y}}_h ) ] }{  \tnorm{ (\underline{\mathbf{W}}_h, \underline{\mathbf{Y}}_h) }_h } &\geq C \tnorm{ ( \underline{\mathbf{U}}_h, \underline{\mathbf{Z}}_h) }_h, \\
	\sup_{ (\underline{\mathbf{W}}_h,\underline{\mathbf{Y}}_h) \in  \ProdFullyDiscrSpace{k}{q} \times \ProdFullyDiscrSpace{ k_{\ast} }{ q_{\ast} }    } \!\!\!\!\!\!\!\!\!\!\!\!\!\!\! \frac{ B[ (\underline{\mathbf{U}}_h,\underline{\mathbf{Z}}_h), ( \underline{\mathbf{W}}_h , \underline{\mathbf{Y}}_h ) ] }{  \tnorm{ (\underline{\mathbf{W}}_h, \underline{\mathbf{Y}}_h) } } &\geq C \tnorm{ ( \underline{\mathbf{U}}_h, \underline{\mathbf{Z}}_h) }. \nonumber 
\end{align} 
 
In the next lemma we record some continuity results for the bilinear form $A[\cdot,\cdot]$.
\begin{lemma}[Continuity]\label{lem:bfi_stab_est}
\begin{enumerate}[label=(\alph*)]
\item For $\underline{\mathbf{U}}_h \in \ProdSemiDiscrSpace{k} + \ProdFullyDiscrSpace{k}{q}$ and $ \mathbf{Y} \in [H^1(\STdom)]^2 + \ProdFullyDiscrSpace{q_{\ast} }{k_{\ast} } $ we have 
\[ A[\underline{\mathbf{U}_h}, \mathbf{Y} ] \leq C \abs{ \underline{\mathbf{U}}_h }_{\underline{S}_h} \left\{ \sum\limits_{n=0}^{N-1} h^{-2} \norm{ y_1 }_{Q^n}^2 +  \norm{ \nabla y_1 }_{Q^n}^2 + \norm{ y_2 }_{Q^n}^2  \right\}^{1/2}. \] 
\item For $\mathbf{U}|_{Q^n} \in \left[ H^{1}(Q^n) \cap L^2(0,T;H^2(\mathcal{T}_h)) \right] \times H^{1}(Q^n) $ for all $n=0,\ldots,N-1$ and $\underline{\mathbf{Y}}_h \in \ProdSemiDiscrSpace{k_{\ast}} + \ProdFullyDiscrSpace{ k_{\ast } }{ q_{\ast} }$ it holds that
\begin{align*}
A[ \mathbf{U}, \underline{\mathbf{Y}}_h ] \leq C & \bigg( \sum\limits_{n=0}^{N-1} \big\{ \norm{\partial_t u_2}_{Q^n}^2  + \norm{ \nabla u_1 }_{Q^n}^2
 + \norm{\partial_t u_1}_{Q^n}^2 + \norm{ u_2}_{Q^n}^2  \\
 & + \int\limits_{I_n} \sum\limits_{K \in \mathcal{T}_h } h^2 \norm{u_1}_{H^2(K)}^2 \; \dT \big\}  \bigg)^{1/2}  \norm{  \underline{\mathbf{Y}}_h }_{ S^{\ast}_h  }.
\end{align*}
\end{enumerate}
\end{lemma}
\begin{proof}
Let us first prove (a). Performing integration by parts in space yields 
\begin{align*}
& \sum\limits_{n=0}^{N-1} \left\{ (\partial_t \underline{u}_2, y_1)_{Q^n} + a( \underline{u}_1,y_1)_{Q^n} - ( \nabla \underline{u}_1 \cdot \mathbf{n}, y_1)_{ \Sigma^n }  
\right\} \\ 
&= \sum\limits_{n=0}^{N-1} \left\{  \int\limits_{I_n} \sum\limits_{K \in \mathcal{T}_h } ( \partial_t \underline{u}_2 - \Delta \underline{u}_1, y_1 )_{K} \; \dT 
	+ \int\limits_{I_n} \sum\limits_{F \in \mathcal{F}_i}   ( \jump{ \nabla \underline{u}_1 \cdot \mathbf{n} } , y_1   )_{F} \; \dT  
	\right\} \\ 
	&\leq C \left[  G( \underline{\mathbf{U}}_h, \underline{\mathbf{U}}_h )^{1/2}  \left( \sum\limits_{n=0}^{N-1} \frac{1}{h^2} \norm{ y_1 }_{Q^n}^2   \right)^{1/2} \!\!\!\!\!\!\!\! 
+ J( \underline{\mathbf{U}}_h, \underline{\mathbf{U}}_h )^{1/2} \left(  \sum\limits_{n=0}^{N-1} \sum\limits_{F \in \mathcal{F}_i} \frac{1}{h} \norm{y_1}_{F}^2 \; \dT  \right)^{1/2}
\right] \\ 
&\leq C \abs{ \underline{\mathbf{U}}_h }_{ S_h } \left( \sum\limits_{n=0}^{N-1} h^{-2} \norm{ y_1 }_{Q^n}^2 +  \norm{ \nabla y_1 }_{Q^n}^2  \right)^{1/2}, 	
\end{align*}
where the trace inequality \cref{ieq:cont-trace-ieq} has been used in the last step. 
Also, 
\[
\sum\limits_{n=0}^{N-1} ( \partial_t \underline{u}_1 - \underline{u}_2, y_2  )_{ Q^n } 
\leq  \sum\limits_{n=0}^{N-1} \norm{ \partial_t \underline{u}_1 - \underline{u}_2 }_{ Q^n } \norm{ y_2 }_{ Q^n }
\leq C \abs{ \underline{\mathbf{U}}_h }_{ S_h } \left( \sum\limits_{n=0}^{N-1} \norm{ y_2 }_{ Q^n }^2 \right)^{1/2}. 
\]
This shows (a). Now to establish (b) we note that
\begin{align*}
\sum\limits_{n=0}^{N-1} \left\{ (\partial_t u_2, \underline{y}_1)_{Q_n} + a( u_1,\underline{y}_1)_{Q^n} \right\}  \leq C \left( \sum\limits_{n=0}^{N-1} \norm{\partial_t u_2}_{Q^n}^2  + \norm{ \nabla u_1 }_{Q^n}^2  \right)^{1/2} \norm{  \underline{\mathbf{Y}}_h }_{ S^{\ast}_h  }.
\end{align*}
Further, by using the trace inequality \cref{ieq:cont-trace-ieq} we obtain  
\begin{align*}
 \sum\limits_{n=0}^{N-1} ( \nabla \underline{u}_1 \!\! \cdot \mathbf{n}, \underline{y}_1)_{ \Sigma^n } 
	& \leq \left( \sum\limits_{n=0}^{N-1} \int\limits_{I_n} \! \sum\limits_{K \in \mathcal{T}_h^n }\!\!\! h \norm{ \nabla u_1 }_{ \partial \Omega \cap \partial K }^2 \! \right)^{\frac{1}{2} } \!\!\!
	\left(  \sum\limits_{n=0}^{N-1} \int\limits_{I_n} \! \sum\limits_{ K \in \mathcal{T}_h^n } \frac{1}{h} \norm{ \underline{y}_1}_{ \partial \Omega \cap \partial K }^2 \right)^{\frac{1}{2} } \\ 
	& \leq C \left(  \sum\limits_{n=0}^{N-1} \big\{  \norm{ \nabla u_1 }_{Q^n}^2  + \int\limits_{I_n} \sum\limits_{ K \in \mathcal{T}_h } h^2 \norm{u_1}_{H^2(K)}^2 \; \dT \big\}  \right)^{1/2}  \norm{  \underline{\mathbf{Y}}_h }_{S^{\ast}_h }.
\end{align*}
An application of the Cauchy-Schwarz inequality then covers the remaining term
\[
\sum\limits_{n=0}^{N-1} ( \partial_t u_1 - u_2, \underline{y}_2  )_{ Q^n } 
\leq C \left( \sum\limits_{n=0}^{N-1} \norm{ \partial_t u_1 }_{ Q^n }^2 + \norm{ u_2 }_{ Q^n }^2 \right)^{1/2} \left( \sum\limits_{n=0}^{N-1} \norm{ \underline{y}_2 }_{ Q^n }^2 \right)^{1/2}.  
\]
\end{proof}
We also require some interpolation operators into the semi and fully-discrete spaces. 
For the proof of the following lemma we refer to the appendix. 
\begin{lemma}[Interpolation]\label{lem:interp-slabwise}
There exist interpolation operators $\Pi_h^k$ into $\SemiDiscrSpace{k}$ and $\Pi_{h, \Delta t}^{k,q}$ into $\FullyDiscrSpace{k}{q}$ such that 
for $\Pi_h \in \{ \Pi_h^k,  \Pi_{h, \Delta t}^{k,q}\}$ under the assumption $\Delta t = C h$ the following interpolation estimates are valid for
$n=0,\ldots,N-1$:
\begin{enumerate}[label=(\alph*)]
\item $\sum\limits_{n=0}^{N-1} \{  h^{-1} \norm{u - \Pi_h u }_{ L^2(Q^n) } + \norm{ \nabla (u - \Pi_h u) }_{ L^2(Q^n) } \} \leq C \norm{u}_{H^1(Q)}  $.		
\item $\sum\limits_{n=0}^{N-1}  \{ h^{-1} \norm{u - \Pi_h u }_{ L^2(Q^n) } + \norm{ u - \Pi_h u }_{ H^1(Q^n) } \} \leq C h^s \norm{u}_{H^{m-1}(Q)}$,
\item For $u \in H^m(Q^n) \cap C^0(I_n, H^2(\Omega))$ we have: 
	\[ \left(\sum\limits_{n=0}^{N-1} \int\limits_{I_n} \sum\limits_{K \in \mathcal{T}_h } h^2 \norm{ u - \Pi_h  u }_{ H^2(K) }^2 \; \dT \right)^{1/2} \leq C h^s \norm{u}_{H^{m}(Q)}, \]
\end{enumerate}
where we have defined 
\begin{equation}\label{def:s-m}
(s,m) = 
\begin{cases} 
    (k,k+2), &  \text{ when } \Pi_h = \Pi_h^k, \\ 
    ( \min\{k,q\}, \max\{k,q\}+3) , & \text{ when } \Pi_h = \Pi_{h, \Delta t}^{k,q}.
\end{cases} 
\end{equation}
\end{lemma}
In the following we continue to use the generic notation $\Pi_h$ to prove some results which are valid for both interpolation operators. Interpolation into the spaces 
containing the dual variable is denoted by $\Pi_h^{\ast} \in \{ \Pi_h^{k_{\ast}},  \Pi_{h, \Delta t}^{k_{\ast},q_{\ast} }\}$. 
These operators are extended to tuples $\mathbf{W} = (w_1,w_2)$ in a componentwise-manner, i.e.\
$\mathbf{\Pi}_h \mathbf{W} :=  ( \Pi_h w_1 , \Pi_h w_2 )$. 
We also keep the assumption $\Delta t = C h$ from now on. 
\begin{lemma}\label{lem:interp_in_stab}
Let $u \in H^{m}(Q)$ solve \cref{eq:PDE+lateralBc} and set $ \mathbf{U} = (u, \partial_t u)$. Let $(s,m)$ be as defined as in \cref{def:s-m}.
\begin{enumerate}[label=(\alph*)]
\item For  $u \in H^{m }(Q) \cap C^{0}([0,T],H^2(\Omega))$ we have 
\[ \abs{ \mathbf{\Pi}_h \mathbf{U} }_{ S_h} = \abs{ \mathbf{\Pi}_h \mathbf{U} - \mathbf{U} }_{ S_h}   \leq C h^{ s }  \norm{u}_{ H^{m}(\STdom)  }.  \]
\item For  $ \underline{\mathbf{U}}_h \in \ProdSemiDiscrSpace{k} + \ProdFullyDiscrSpace{k}{q}$ it holds that
\[
 \norm{u-\underline{u}_1}_{\Sigma}  + \norm{u-\underline{u}_1}_{\STdata}
\leq C \left(  h^{ s+1/2 }  \norm{u}_{ H^{m}(\STdom)} + \tnorm{ ( \underline{\mathbf{U}}_h - \mathbf{\Pi}_h \mathbf{U},0) }_h \right).
\]
\item For $\mathbf{Y} = (y_1,y_2) \in [H_0^1(\STdom )]^2$ we have
\[ \norm{ \mathbf{\Pi}_h^{\ast}  \mathbf{Y} }_{ S^{\ast}_h } \leq C \left( \norm{y_1}_{H^1(\STdom)} + \norm{y_2}_{H^1(\STdom)} \right). 
\]
\end{enumerate}
\end{lemma}
\begin{proof}
\begin{enumerate}[label=(\alph*)]
\item We estimate the terms in $\abs{ \mathbf{\Pi}_h \mathbf{U} }_{ S_h}$ separately. \par 
\noindent By using smoothness of $u$, the trace inequality \cref{ieq:cont-trace-ieq} and the interpolation estimates from \cref{lem:interp-slabwise} we obtain
\begin{align*}
& J( \mathbf{\Pi}_h \mathbf{U} , \mathbf{\Pi}_h \mathbf{U} ) = \sum\limits_{n=0}^{N-1} \int\limits_{I_n} \sum\limits_{F \in \mathcal{F}_i} h \norm{\jump{ \nabla ( \Pi_h u  - u)}}_{F}^2 \; \dT \\  
&\leq C \sum\limits_{n=0}^{N-1} \int\limits_{I_n} \sum\limits_{ K \in \mathcal{T}_h }  \left(   \norm{  \nabla ( \Pi_h u - u ) }^2_{ L^2(K)  } + h^2 \norm{  \Pi_h u - u  }^2_{ H^2(K)  }  \right) \dT \\
	& \leq C h^{2 s }  \norm{u}_{ H^m(\STdom)  }^2. 
\end{align*}
Inserting $\partial_t^2 u - \Delta u = 0$ and applying interpolation bounds gives
\begin{align*}
G( \mathbf{\Pi}_h \mathbf{U} ,\mathbf{\Pi}_h \mathbf{U} ) &= \sum\limits_{n=0}^{N-1} \int\limits_{I_n} \sum\limits_{K \in \mathcal{T}_h^n } h^2 \norm{ \partial_t \left( \Pi_h \partial_t u - \partial_t u \right) + \Delta ( u - \Pi_h  u) }_{K}^2  \; \dT \\
&  \leq C h^{2s} \left( \norm{\partial_t u }_{H^{m-1}(\STdom) }^2  + \norm{u}_{H^m(\STdom)}^2  \right) \leq C h^{2s} \norm{u}_{H^m(\STdom)}^2.  
\end{align*}
Similarly, we estimate
\begin{align*}
 I_0( \mathbf{\Pi}_h  \mathbf{U} , \mathbf{\Pi}_h  \mathbf{U} ) & =  \sum\limits_{n=0}^{N-1} \norm{  ( \Pi_h  \partial_t u - \partial_t u) +  \partial_t(u - \Pi_h  u)    }^2_{ Q^n } \leq C h^{2s} \norm{u}_{H^m(\STdom)}^2.
\end{align*}
The boundary term is treated by using the trace inequality and then appealing to the interpolation results: 
\begin{align*}
&  R( \mathbf{\Pi}_h \mathbf{U} , \mathbf{\Pi}_h \mathbf{U} )  =  \sum\limits_{n=0}^{N-1} \int\limits_{I_n} \sum\limits_{K \in \mathcal{T}_h } h^{-1} \norm{ \Pi_h u - u  }_{ L^2(\partial \Omega \cap \partial K )  }^2    \dT \\
	\leq &  C  \sum\limits_{n=0}^{N-1} \int\limits_{I_n} \sum\limits_{K \in \mathcal{T}_h } \left( h^{-2}  \norm{ \Pi_h u - u}_{ L^2(K ) }^2 + \norm{ \nabla \left( \Pi_h  u - u \right) }_{ L^2(K ) }^2    \right)  \dT \\
	& \leq C h^{2s} \norm{u}_{H^{m-1} (\STdom)}^2. 
\end{align*}
\item From interpolation estimates and the definition \cref{eq:triple_norm_h_def} of $\tnorm{\cdot}_h$ we obtain 
\begin{align*}
\norm{u-\underline{u}_1}_{\STdata} & \leq C \left( \norm{u-\Pi_h u }_{\STdata} + \tnorm{ ( \mathbf{\Pi}_h \mathbf{U} - \underline{ \mathbf{U}}_h , 0) }_h  \right)  \\
& \leq C \left(  h^{ s+1 }  \norm{u}_{ H^{m}(\STdom)} + \tnorm{ ( \underline{  \mathbf{U}}_h  - \mathbf{\Pi}_h \mathbf{U}  , 0) }_h  \right). 
\end{align*}
We proceed similarly for the boundary term using the trace inequality:  
\begin{align*}
& \norm{u-\underline{u}_1}_{\Sigma}^2   
	\leq C  \sum\limits_{n=0}^{N-1} \int\limits_{I_n} \sum\limits_{K \in \mathcal{T}_h }  \norm{ u - \Pi_h u }_{ L^2(\partial \Omega \cap \partial K )  }^2 +  \frac{1}{h} \norm{ \Pi_h u - \underline{u}_1  }_{ L^2(\partial \Omega \cap \partial K )  }^2   \\
& \leq  C  \left(  \sum\limits_{n=0}^{N-1} \left[ \frac{1}{h} \norm{  u - \Pi_h u }^2_{Q^n} +  h \norm{ \nabla (  u - \Pi_h u ) }^2_{Q^n}   \right] +  \tnorm{ ( \underline{  \mathbf{U}}_h  - \mathbf{\Pi}_h \mathbf{U}  , 0) }_h^2  \right) \\
& \leq C \left(  h^{ 2(s+1/2) }  \norm{u}_{ H^{m}(\STdom)}^2  + \tnorm{ ( \underline{  \mathbf{U}}_h  - \mathbf{\Pi}_h \mathbf{U}  , 0) }_h^2 \right). 
\end{align*}
\item By applying \cref{lem:interp-slabwise} (a) we obtain   
\[
\norm{ \Pi_h^{\ast} y_1  }_{ Q^n } +  \norm{ \Pi_h^{\ast} y_2  }_{ Q^n } + \norm{ \nabla \Pi_h^{\ast}  y_1  }_{ Q^n }  
	\leq C  \left( \norm{y_1}_{H^1(Q^n)} + \norm{y_2}_{H^1(Q^n)}  \right). 
\]
Using that $y_1$, vanishes on the lateral boundary we obtain by means of the trace inequality and \cref{lem:interp-slabwise} (a):
\begin{align*}
& h^{-1} \norm{ \Pi_h^{\ast} y_1}_{\Sigma^n}^2 = h^{-1} \norm{ \Pi_h^{\ast} y_1 - y_1}_{\Sigma^n}^2 \\
& \leq C \int\limits_{I_n} \left( h^{-2} \norm{ \Pi_h^{\ast} y_1 - y_1  }_{ L^2(\Omega ) }^2 + \norm{ \Pi_h^{\ast} y_1 - y_1  }_{ H^1(\Omega ) }^2  \right)  \; \dT \leq C \norm{y_1}_{H^1(Q^n)}^2. 
\end{align*}
\end{enumerate} 
\end{proof}

\subsection{Consistency and residual bounds}\label{ssection:consistency-res}
In this subsection we establish two additional results which are crucial for the analysis of the semi-discrete and the fully-discrete method. The first lemma mainly follows from consistency of the stabilization $S_h$ and the bilinear form $A$. 

\begin{lemma}\label{lem:consistency-bound}
Let $u$ be a solution of \cref{eq:PDE+lateralBc} and set $ \mathbf{U} = (u, \partial_t u)$. Then for any  
$ \underline{\mathbf{W}}_h \in \ProdSemiDiscrSpace{k} + \ProdFullyDiscrSpace{k}{q}$ and $ \underline{\mathbf{Y}}_h \in \ProdSemiDiscrSpace{k_{\ast}} + \ProdFullyDiscrSpace{ k_{\ast } }{ q_{\ast} }$ it holds that 
\begin{align*}
(u -  \Pi_h u,\underline{w}_1)_{\STdata} -  S_h(\mathbf{\Pi}_h \mathbf{U}, \underline{\mathbf{W}}_h) -
A[ \mathbf{\Pi}_h \mathbf{U} , \underline{\mathbf{Y}}_h ] 
\leq C h^{s} \norm{u}_{H^m(\STdom)} \tnorm{ (\underline{\mathbf{W}}_h, \underline{\mathbf{Y}}_h) }_h, 
\end{align*}
where $s$ is as defined in \cref{def:s-m}.
\end{lemma}
\begin{proof}
From \cref{lem:interp_in_stab} (a) and interpolation results we immediately have 
\begin{align*}
(u -  \Pi_h u,\underline{w}_1)_{\STdata} - S_h(\mathbf{\Pi}_h \mathbf{U}, \underline{\mathbf{W}}_h) 
	& \leq C \left( \norm{ u -  \Pi_h u}_{ \STdata} + \abs{ \mathbf{\Pi}_h \mathbf{U}}_{S_h}  \right)  \tnorm{ (\underline{\mathbf{W}}_h, 0) }_h \\
	& \leq C h^{s} \norm{u}_{H^m(\STdom)}  \tnorm{ (\underline{\mathbf{W}}_h, 0) }_h. 
\end{align*}
For the remaining term we use $ A[ \mathbf{U} , \underline{\mathbf{Y}}_h ] = 0 $, \cref{lem:bfi_stab_est} (b) and interpolation estimates from 
\cref{lem:interp-slabwise} (b)-(c) to obtain
\begin{align*}
& A[ \mathbf{\Pi}_h \mathbf{U} , \underline{\mathbf{Y}}_h ] = A[ \mathbf{\Pi}_h \mathbf{U} - \mathbf{U}, \underline{\mathbf{Y}}_h ]  \\ 
	& \leq C \left( \sum\limits_{n=0}^{N-1}  \sum\limits_{i=1,2} \left\{ \norm{ \Pi_h u_i - u_i }_{ H^1(Q^n) }^2 \right\} + \int\limits_{I_n} \sum\limits_{K \in \mathcal{T}_h } h^2 \norm{  \Pi_h  u_1 - u_1  }_{ H^2(K) }^2   \right)^{1/2} 
\norm{  \underline{\mathbf{Y}}_h }_{ S^{\ast}_h  } \\ 
& \leq C h^{s} \norm{u}_{H^m(\STdom)}  \norm{  \underline{\mathbf{Y}}_h }_{ S^{\ast}_h  }.
\end{align*}
\end{proof} 

The second lemma will later allow us to control the PDE residual $\norm{\Box \underline{u}_1 }_{H^1(\STdom)}$. 
 \begin{lemma}\label{lem:res-bound}
Let $ ( \underline{\mathbf{U}}_h, \underline{\mathbf{Z}}_h, \mathbf{\Pi}_h^{\ast} ) \in   \{( \ProdSemiDiscrSpace{k},\ProdSemiDiscrSpace{k_{\ast}},\mathbf{\Pi}_h^{k_{\ast}}) ,(\ProdFullyDiscrSpace{k}{q}, \ProdFullyDiscrSpace{ k_{\ast } }{ q_{\ast} }, \mathbf{\Pi}_{h, \Delta t}^{k_{\ast},q_{\ast}}) \}$ such that 
\begin{equation}\label{eq:optimality_dual}
A[ \underline{\mathbf{U}}_h, \mathbf{\Pi}_h^{\ast} \mathbf{W} ] = S_h^{\ast}(  \mathbf{\Pi}_h^{\ast} \mathbf{W}  , \underline{\mathbf{Z}}_h)  
\end{equation}
for any $ \mathbf{W} = (w_1,w_2) = (w,0)$ with $w \in H^1_0(\STdom)$ holds true. Then for  $ \mathbf{U} = (u, \partial_t u)$ with $u$ being the solution of \cref{eq:PDE+lateralBc} we have: 
\begin{align*}
& \sum\limits_{n=0}^{N-1} \{ (\partial_t \underline{u}_1, \partial_t w_1)_{ Q^n } - a(\underline{u}_1,w_1)_{Q^n} \} \\ 
& \leq \left\vert \sum\limits_{n=1}^{N-1} ( \jump{\underline{u}_2^n}, w_{1,+}^n  )_{\Omega} \right\vert + C \norm{w}_{H^1(\STdom)} \left(  \tnorm{ ( \underline{\mathbf{U}}_h - \mathbf{\Pi}_h \mathbf{U} ,\underline{\mathbf{Z}}_h )   }_h  +  \abs{ \mathbf{\Pi}_h \mathbf{U} }_{ S_h}  \right). 
\end{align*}
\end{lemma}
\begin{proof}
Integration by parts in time yields
\begin{align*}
& \sum\limits_{n=0}^{N-1} \{ (\partial_t \underline{u}_1, \partial_t w_1)_{ Q^n } - a(\underline{u}_1,w_1)_{Q^n} \} 
 = \sum\limits_{n=0}^{N-1} (\partial_t \underline{u}_1 - \underline{u}_2, \partial_t w_1)_{ Q^n } -  \sum\limits_{n=1}^{N-1} ( \jump{\underline{u}_2^n}, w_{1,+}^n  )_{\Omega} \\ 
& - \sum\limits_{n=0}^{N-1} \left\{  (\partial_t \underline{u}_2 , w_1)_{ Q^n } +  a(\underline{u}_1,w_1)_{Q^n} +  (\partial_t \underline{u}_1 - \underline{u}_2 , \underbrace{ w_2 }_{=0} )_{ Q^n } - (\nabla \underline{u}_1 \cdot \mathbf{n},  \underbrace{w_1}_{ = 0  } )_{ \Sigma^n } \right\}  \\ 
& = \sum\limits_{n=0}^{N-1} (\partial_t \underline{u}_1 - \underline{u}_2, \partial_t w_1)_{ Q^n } -  \sum\limits_{n=1}^{N-1} ( \jump{\underline{u}_2^n}, w_{1,+}^n  )_{\Omega} 
- A[\underline{\mathbf{U}_h},  \mathbf{W} ].  
\end{align*}
We treat the terms separately. Clearly,    
\[
\sum\limits_{n=0}^{N-1} (\partial_t \underline{u}_1 - \underline{u}_2, \partial_t w_1)_{ Q^n } 
\leq \underline{I}_0( \underline{\mathbf{U}}_h, \underline{\mathbf{U}}_h )^{1/2} \norm{\partial_t w_1  }_{\STdom} 
\leq C \abs{  \underline{\mathbf{U}}_h   }_{ S_h } \norm{w}_{H^1(\STdom)}. 
\]
To cope with the remaining term we insert $ \mathbf{\Pi}_h^{\ast}  \mathbf{W} = (  \Pi_h^{\ast} w,0) $ so that 
\bel{eq:disc_bil_splitting}
A[\underline{\mathbf{U}}_h, \mathbf{W} ] = A[\underline{\mathbf{U}}_h, \mathbf{W} -  \mathbf{\Pi}_h^{\ast} \mathbf{W} ] + A[ \underline{\mathbf{U}}_h, \mathbf{\Pi}_h^{\ast} \mathbf{W}]. 
\ee
From \cref{lem:bfi_stab_est} (a) and \cref{lem:interp-slabwise} (a)  we obtain
\begin{align*}
\underline{A}[\underline{\mathbf{U}}_h, \mathbf{W} - \mathbf{\Pi}_h^{\ast} \mathbf{W} ] 
	& \leq C \abs{ \underline{\mathbf{U}}_h }_{ S_h } \bigg\{ \sum\limits_{n=0}^{N-1} \frac{1}{h^2} \norm{ w_1 - \Pi_h^{\ast} w_1 }_{Q^n}^2 + \norm{ \nabla (w_1 - \Pi_h^{\ast} w_1 )}_{Q^n}^2  \bigg\}^{1/2} \\
 & \leq C  \abs{ \underline{\mathbf{U}}_h }_{ S_h } \norm{w}_{H^1(\STdom)}.
\end{align*}
For the other contribution we obtain from \cref{eq:optimality_dual} and \cref{lem:interp_in_stab} (c) that 
\[
A[\underline{\mathbf{U}}_h,  \mathbf{\Pi}_h^{\ast} \mathbf{W} ] = \underline{S}^{\ast} ( \mathbf{\Pi}_h^{\ast} \mathbf{W}, \underline{\mathbf{Z}}_h) 
\leq C \norm{ \mathbf{\Pi}_h^{\ast} \mathbf{W}}_{ S^{\ast}_h }  \norm{ \underline{\mathbf{Z}}_h }_{ S^{\ast}_h } 
\leq C  \norm{ \underline{\mathbf{Z}}_h }_{ S^{\ast}_h } \norm{w}_{H^1(\STdom)}.
\]
The claim follows in view of  
\[
 \abs{ \underline{\mathbf{U}}_h }_{ S_h } + \norm{ \underline{\mathbf{Z}}_h }_{ S^{\ast}_h } 
\leq \tnorm{ ( \underline{\mathbf{U}}_h - \mathbf{\Pi}_h \mathbf{U} ,\underline{\mathbf{Z}}_h )   }_h + \abs{ \mathbf{\Pi}_h \mathbf{U} }_{ S_h}. 
\]
\end{proof}

\section{Error analysis for the semi-discrete method}\label{section:error-analysis-semi-discrete}
In this section we derive an error estimate for the semi-discrete method defined in \cref{eq:opt_compact_semi}. Thanks to the auxiliary results derived in \cref{ssection:stab-interp}-
\cref{ssection:consistency-res} and the fact that no jumps over time slice boundaries occur, we easily obtain the desired result. First we deduce convergence of the discretization error in the triple norm $\tnorm{ (\cdot , \cdot ) }_h$. 
\begin{lemma}\label{lem:triple_norm_semi_disc_conv}
Let $u$ be a sufficiently regular solution of \cref{eq:PDE+lateralBc} and set $\mathbf{U} := (u,\partial_t u) $.
Let  $ (\underline{\mathbf{U}}_h ,\underline{\mathbf{Z}}_h) \in  \ProdSemiDiscrSpace{k} \times \ProdSemiDiscrSpace{k_{\ast}}  $ 
be the solution of \cref{eq:opt_compact_semi}. Then 
\[
\tnorm{ (\underline{\mathbf{U}}_h - \mathbf{\Pi}_h \mathbf{U} , \underline{\mathbf{Z}}_h ) }_h \leq C h^{k}  \norm{u}_{ H^{k+2}(\STdom)  }. 
\]
\end{lemma}
\begin{proof}
From \cref{eq:opt_compact_semi} we have 
\[
B_h[ ( \underline{\mathbf{U}}_h - \mathbf{\Pi}_h \mathbf{U}, \underline{\mathbf{Z}}_h ), ( \underline{\mathbf{W}}_h ,\underline{\mathbf{Y}}_h ) ] 
= (u -  \Pi_h u,\underline{w}_1)_{\STdata} -  S_h(\mathbf{\Pi}_h \mathbf{U}, \underline{\mathbf{W}}_h) -
 A[ \mathbf{\Pi}_h \mathbf{U} , \underline{\mathbf{Y}}_h]
\]
for any $ (\underline{\mathbf{W}}_h ,\underline{\mathbf{Y}}_h) \in  \ProdSemiDiscrSpace{k} \times \ProdSemiDiscrSpace{k_{\ast}} $ 
Hence, the claim follows from \cref{lem:consistency-bound} and the inf-sup condition \cref{eq:inf-sup} for $B_h$. 
\end{proof}
We then improve the former result to obtain convergence in a physically significant norm by utilizing the Lipschitz stability result from \cref{thm:Lipschitz-stab}.
\begin{theorem}\label{thm:error-estimate-semi-discrete}
Let $u$ be a sufficiently regular solution of \cref{eq:PDE+lateralBc} and $ (\underline{\mathbf{U}}_h ,\underline{\mathbf{Z}}_h) \in  \ProdSemiDiscrSpace{k} \times \ProdSemiDiscrSpace{k_{\ast}}  $ be the solution of \cref{eq:opt_compact_semi}. Then 
\[
\norm{ u - \underline{u}_1 }_{L^{\infty}(0,T;L^2(\Omega))} + \norm{\partial_t (u-\underline{u}_1) }_{L^{2}(0,T;H^{-1}(\Omega))} 
\leq C h^{k} \norm{u}_{ H^{k+2}(\STdom)  }.
\]
\end{theorem}
\begin{proof}
Since $e_h := u - \underline{u}_1 \in H^1(\STdom)$, we can directly apply the stability estimate \cref{eq:Lipschitz-stability} to $\phi = e_h $ to obtain
\begin{align*}
\norm{ e_h }_{L^{\infty}(0,T;L^2(\Omega))} + \norm{\partial_t e_h }_{L^{2}(0,T;H^{-1}(\Omega))} \! 
\leq C \! \left(  \norm{ e_h }_{ \STdata } \! + \! \norm{ e_h }_{ L^2(\Sigma)} \! + \! \norm{ \Box e_h }_{ H^{-1} ( \STdom ) } \right) \!.
\end{align*}
From \cref{lem:interp_in_stab} (b) and \cref{lem:triple_norm_semi_disc_conv} we see that the first two terms on the right hand side are bounded 
by a constant times $ h^{k}  \norm{u}_{ H^{k+2}(\STdom)  }$. The same bound is valid for the last term, which we deduce by writing down the definition
\begin{align*}
\norm{ \Box e_h }_{ H^{-1} ( \STdom ) } = \norm{ - \Box \underline{u}_1}_{ H^{-1} ( \STdom ) } =  \sup_{\substack{w \in H^1_0(\STdom), \\ \norm{w}_{H^1(\STdom) } = 1 }} \sum\limits_{n=0}^{N-1} \{ (\partial_t \underline{u}_1, \partial_t w_1)_{ Q^n } - a(\underline{u}_1,w_1)_{Q^n} \}  
\end{align*}
and applying \cref{lem:res-bound}\footnote{Note that \cref{eq:optimality_dual} holds because of the first order optimality condition \cref{eq:optimality_cond_semi}.}, \cref{lem:triple_norm_semi_disc_conv}, \cref{lem:interp_in_stab} (a) and the fact that $ \jump{\underline{u}_2^n} = 0$ for $\underline{u}_2 \in \SemiDiscrSpace{k}$ holds. 
\end{proof}

\section{Error analysis for the fully-discrete method}\label{section:error-analysis-fully-discrete}
In this section we derive an error estimate for the fully-discrete method defined in \cref{eq:opt_compact_fully_disc}. In view of the results that have already been established in \cref{section:error-analysis-semi-discrete} for the semi-discrete method, this basically amounts to treating the additional coupling terms that were introduced by adding the term $S^{\uparrow \downarrow}_{\Delta t}$ on top of the semi-discrete Lagrangian.
As indicated in \cref{def:s-m} we will set $(s,m) = ( \min\{k,q\}, \max\{k,q\}+3)$ and hold on to the assumption $\Delta t = C h$ in this section. \par 
First we require some additional interpolation results at the boundary of the time slices whose proof can be found in the appendix.
\begin{lemma}[Approximation bounds at fixed time levels]\label{lem:approx-fixed-time-lvl}
At the top and bottom of the time slice the following approximation bounds for the traces holds true for $n=0,\ldots,N-1$:
\begin{enumerate}[label=(\alph*)]
\item  $ \norm{(u - \Pi_h^{\ast} u )( \tau, \cdot )}_{\Omega} \leq C h^{1/2} \norm{u}_{H^{1}(Q^n)},  \; \tau \in \{ t_n,t_{n+1}\}$,
\item  $ \norm{(u - \Pi_h u )( \tau, \cdot )}_{\Omega} \leq C h^{s+1/2} \norm{u}_{H^{m-2}(Q^n)},  \; \tau \in \{ t_n,t_{n+1}\}$,
\item $  \norm{\nabla(u - \Pi_h u )( \tau, \cdot )}_{\Omega} \leq C h^{s-1/2} \norm{u}_{H^{m-1}(Q^n)},  \; \tau \in \{ t_n,t_{n+1}\}$.
\end{enumerate}
\end{lemma}
We then deduce the counterpart of \cref{lem:interp_in_stab} (a) for the additional stabilization term. 
\begin{lemma}\label{lem:interp-in-jump-over-time-slice-bnd}
Let $u$ be a sufficiently regular solution of \cref{eq:PDE+lateralBc} and set $\mathbf{U} := (u,\partial_t u) $. Then 
\[
\abs{ \mathbf{\Pi}_h \mathbf{U} }_{ \uparrow \downarrow } \leq C  h^{s} \norm{u}_{H^{m-1}(\STdom)}.
\]
\end{lemma}
\begin{proof}
By using \cref{lem:approx-fixed-time-lvl} and smoothness of $u$ we obtain 
\begin{align*}
& I_2^{\uparrow \downarrow}(  \mathbf{\Pi}_h \mathbf{U}, \mathbf{\Pi}_h \mathbf{U}) = \sum\limits_{n=1}^{N-1} \frac{1}{\Delta t} \norm{ \jump{ (\Pi_h \partial_t u - \partial_t u)^n  } }_{\Omega}^2 \\ 
& \leq \frac{2}{\Delta t } \left(  \sum\limits_{n=1}^{N-1} \norm{  (\Pi_h \partial_t u - \partial_t u)(t_n,\cdot)   }_{\Omega}^2 + \sum\limits_{n=0}^{N-2} \norm{  (\Pi_h \partial_t u - \partial_t u)(t_{n+1},\cdot)   }_{\Omega}^2    \right) \\ 
& \leq C h^{2s} \norm{u}_{H^{m-1}(\STdom)}^2.
\end{align*}
In a similar way we bound the contributions in $I_1$, e.g.\
\begin{align*}
& \sum\limits_{n=1}^{N-1} \Delta t \norm{ \jump{ \nabla (\Pi_h u - u)^n  } }_{\Omega}^2  \leq 2 \Delta t \Big(  \sum\limits_{n=1}^{N-1} \norm{ \nabla (\Pi_h u - u)(t_n,\cdot)   }_{\Omega}^2 \\ 
& + \sum\limits_{n=0}^{N-2} \norm{ \nabla (\Pi_h u - u)(t_{n+1},\cdot)   }_{\Omega}^2  \Big) 
 \leq C h^{2s} \norm{u}_{H^{m-1}(\STdom)}^2.
\end{align*}
\end{proof}
With this result, convergence of the discretization error in the strengthened triple norm $\tnorm{(\cdot,\cdot)}$ immediately follows.
\begin{lemma}\label{lem:triple_norm_fully_disc_conv}
Let $u$ be a sufficiently regular solution of \cref{eq:PDE+lateralBc} and set $\mathbf{U} := (u,\partial_t u) $.
Let  $ (\underline{\mathbf{U}}_h ,\underline{\mathbf{Z}}_h) \in \ProdFullyDiscrSpace{k}{q} \times \ProdFullyDiscrSpace{ k_{\ast} }{ q_{\ast} } $ 
be the solution of \cref{eq:opt_compact_fully_disc}. Then 
\[
\tnorm{ (\underline{\mathbf{U}}_h - \mathbf{\Pi}_h \mathbf{U} , \underline{\mathbf{Z}}_h ) } \leq C h^{s}  \norm{u}_{ H^{m}(\STdom)  }. 
\]
\end{lemma}
\begin{proof}
From \cref{eq:complete_bfi_def_fully_discr} we obtain 
\begin{align*}
B[ ( \underline{\mathbf{U}}_h - \mathbf{\Pi}_h \mathbf{U}, \underline{\mathbf{Z}}_h ), ( \underline{\mathbf{W}}_h ,\underline{\mathbf{Y}}_h ) ] 
= &  (u -  \Pi_h u,\underline{w}_1)_{\STdata} - S_h(\mathbf{\Pi}_h \mathbf{U}, \underline{\mathbf{W}}_h) -
 A[ \mathbf{\Pi}_h \mathbf{U} , \underline{\mathbf{Y}}_h] \\
	& \!\! - S^{\uparrow \downarrow}_{\Delta t}( \mathbf{\Pi}_h \mathbf{U} , \underline{\mathbf{W}}_h ).
\end{align*}
Since $S^{\uparrow \downarrow}_{\Delta t}( \mathbf{\Pi}_h \mathbf{U} , \underline{\mathbf{W}}_h ) \leq C h^{s} \norm{u}_{H^{m-1}(\STdom)} \abs{ \underline{\mathbf{W}}_h }_{ \uparrow \downarrow }  $ thanks to \cref{lem:interp-in-jump-over-time-slice-bnd}, the claim follows as in \cref{lem:triple_norm_semi_disc_conv}. 
\end{proof}
In the semi-discrete case we were able to apply the Lipschitz stability result \cref{eq:Lipschitz-stability} directly to the error $u - \underline{u}_1$ to deduce our error estimate. In the fully-discrete case, this is no longer possible since $\underline{u}_1 \notin H^1(\STdom)$. To address this issue, we follow \cite{BDE23} and introduce a lifting operator $L_{\Delta t}:\FullyDiscrSpace{k}{q} \rightarrow C^0((0,T);V_h^k)$ as follows. 
Let $\vartheta_{n}(t) := (t_{n+1} - t)/\Delta t$  for $ t \in I_n$ and set 
\begin{equation}\label{eq:def_lifting}
L_{\Delta t} \underline{w}(t) = \underline{w}(t) - \jump{ \underline{w}^n }  \vartheta_{n}(t), \quad  t \in I_n, n \geq 1, \quad
	L_{\Delta t} \underline{w}(t) = \underline{w}(t),  \quad t \in I_0.
\end{equation}
Note that $\vartheta_{n}(t_n) = 1$ and $ \vartheta_{n}(t_{n+1} ) =  0$,  
which implies 
\[
\lim_{ t \downarrow t_n} L_{\Delta t} \underline{w}(t)  
= \underline{w}_{+}^{n} - \jump{ \underline{w}^n }
= \underline{w}_{-}^{n} 
= \lim_{ t \uparrow t_n} L_{\Delta t} \underline{w}(t).
\]
Thus, $ L_{\Delta t} \underline{w}$ is continuous in time. For later purposes we also record that 
\begin{equation}\label{eq:theta-and-deriv-L2-norm}
\norm{ \vartheta_{n}}_{ L^2(I_n) } = ( \Delta t)^{1/2}/\sqrt{3},
\qquad 
\norm{ \vartheta_{n}^{\prime} }_{ L^2(I_n) } =  ( \Delta t)^{-1/2}. 
\end{equation}
For the smoothed solution $L_{\Delta t} \underline{u}_1$ convergence rates follow similarly as in \cref{thm:error-estimate-semi-discrete}, where the additional jump terms induced by the lifting \cref{eq:def_lifting} are covered by the stabilization term $S^{\uparrow \downarrow}_{\Delta t}  $. Note that the convergence order is independent of the polynomial order of the dual variable. 
\begin{theorem}\label{thm:error-estimate-fully-discrete}
Let $u$ be a sufficiently regular solution of \cref{eq:PDE+lateralBc} and $ (\underline{\mathbf{U}}_h ,\underline{\mathbf{Z}}_h) \in \ProdFullyDiscrSpace{k}{q} \times \ProdFullyDiscrSpace{ k_{\ast} }{ q_{\ast} } $ be the solution of \cref{eq:opt_compact_fully_disc}. 
Then for $(s,m) = ( \min\{k,q\}, \max\{k,q\}+3)$ it holds
\[
\norm{ u - L_{\Delta t} \underline{u}_1 }_{L^{\infty}(0,T;L^2(\Omega))} + \norm{\partial_t (u - L_{\Delta t} \underline{u}_1) }_{L^{2}(0,T;H^{-1}(\Omega))} 
\leq C h^{s} \norm{u}_{ H^{m}(\STdom)  }.
\]
\end{theorem}
\begin{proof}
Since $L_{\Delta t} \underline{u}_1 \in H^1(Q)$, we are allowed to apply the continuum stability estimate \cref{eq:Lipschitz-stability} to
$\tilde{e}_h := u - L_{\Delta t} \underline{u}_1$ which yields: 
\begin{align}
& \norm{ \tilde{e}_h }_{L^{\infty}(0,T;L^2(\Omega))} + \norm{\partial_t \tilde{e}_h }_{L^{2}(0,T;H^{-1}(\Omega))} \nonumber \\ 
& \leq C \left(  \norm{ e_h }_{ \STdata } + \norm{ e_h }_{ \Sigma}
        +  \norm{ \underline{u}_1 -  L_{\Delta t} \underline{u}_1 }_{ \STdata } + \norm{ \underline{u}_1 -  L_{\Delta t} \underline{u}_1 }_{ \Sigma }
	+ \norm{ \Box \tilde{e}_h }_{ H^{-1} ( \STdom ) } \right), \label{eq:appl-stability-lifted-error}
\end{align}
where $e_h := u - \underline{u}_1$ as in the proof of \cref{thm:error-estimate-semi-discrete}.
We need to control the terms on the right hand side. Starting with $\norm{ \Box \tilde{e}_h }_{ H^{-1} ( \STdom ) }$, we have for any $w \in H_0^1(\STdom)$
in view of $\Box u = 0$: 
\begin{align*}
	(- \partial_t \tilde{e}_h, \partial_t w)_Q + a(\tilde{e}_h,w)_Q 
& =   \sum\limits_{n=0}^{N-1} \{ (\partial_t \underline{u}_1, \partial_t w_1)_{ Q^n } - a(\underline{u}_1,w_1)_{Q^n} \} + \xi( \underline{u}_1, w) \\
& \leq C h^{s}  \norm{u}_{ H^{m}(\STdom)  } \norm{w}_{H^1(\STdom)} +  \xi( \underline{u}_1, w) + \eta( \underline{u}_2, w), 
\end{align*}
where \cref{lem:res-bound}, \cref{lem:triple_norm_fully_disc_conv} and \cref{lem:interp_in_stab} (a) have been used and  
\begin{align*}
\xi( \underline{u}_1, w) :=  \sum\limits_{n=1}^{N-1} \{ a( \vartheta_{n} \jump{ \underline{u}_1^n } , w  )_{Q^n}  -( \vartheta^{\prime}_{n} \jump{ \underline{u}_1^n } , \partial_t w  )_{Q^n} \}, \;  
\eta( \underline{u}_2, w) := \left\vert \sum\limits_{n=1}^{N-1} ( \jump{\underline{u}_2^n}, w_{+}^n  )_{\Omega} \right\vert. 
\end{align*}
By using \cref{eq:theta-and-deriv-L2-norm} we obtain
\begin{align*}
& \xi( \underline{u}_1, w) \leq \left(  \sum\limits_{n=1}^{N-1} \norm{ \vartheta^{\prime}_{n}  }^2_{L^2(I_n)} \norm{ \jump{ \underline{u}_1^n } }_{ \Omega}^2 
+ \norm{ \vartheta_{n}  }^2_{L^2(I_n)} \norm{ \jump{ \nabla \underline{u}_1^n } }_{ \Omega}^2  \right)^{1/2} \norm{ w }_{ H^1(\STdom) }  \\  
& \leq C \left(  \sum\limits_{n=1}^{N-1}   (\Delta t)^{-1}  \norm{ \jump{ \underline{u}_1^n } }_{ \Omega}^2 
+ \Delta t \norm{ \jump{ \nabla \underline{u}_1^n } }_{ \Omega}^2  \right)^{1/2} \norm{ w }_{ H^1(\STdom ) } 
 \leq C \abs{ \underline{\mathbf{U}}_h }_{ \uparrow \downarrow } \norm{ w }_{ H^1(\STdom) }. 
\end{align*}
To treat $\eta( \underline{u}_2, w)$, we first note that 
\[
\norm{w_{+}^n }_{\Omega} \leq \norm{ w_{+}^n - \Pi_h^{\ast} w_{+}^n  }_{\Omega} + \norm{  \Pi_h^{\ast} w_{+}^n  }_{\Omega} 
	\leq C \left( h^{1/2} + (\Delta t)^{-1/2}   \right) \norm{w}_{H^{1}(Q^n)} ,
\]
	where we applied \cref{lem:approx-fixed-time-lvl} and a discrete inverse inequality in time (see \cref{Lemma:Discrete inverse inequality in time}) to estimate the second term.
Hence, 
\[
\eta( \underline{u}_2, w) \leq C \left(  \sum\limits_{n=1}^{N-1} \frac{1}{\Delta t} \norm{  \jump{\underline{u}_2^n} }_{\Omega}^2 \right)^{1/2}
\left(  \sum\limits_{n=1}^{N-1} \Delta t \norm{w_{+}^n}_{\Omega}^2 \right)^{1/2}
\leq C \abs{ \underline{\mathbf{U}}_h }_{ \uparrow \downarrow } \norm{ w }_{ H^1(\STdom) }. 
\]
From \cref{lem:interp-in-jump-over-time-slice-bnd} and \cref{lem:triple_norm_fully_disc_conv} we obtain that $\abs{ \underline{\mathbf{U}}_h }_{ \uparrow \downarrow } \leq C h^{s}  \norm{u}_{ H^{m}(\STdom)  }$  and hence we conclude that the same bound holds for $\norm{ \Box \tilde{e}_h }_{ H^{-1} ( \STdom ) }$. 
The terms $ \norm{ e_h }_{ L^2(\STdata) }$ and $ \norm{ e_h }_{ L^2(\Sigma)}$ obey this bound as well which follows similarly as in the proof 
of \cref{thm:error-estimate-semi-discrete} for the semi-discrete case. 
It remains to treat
\begin{align*}
& \norm{ \underline{u}_1 -  L_{\Delta t} \underline{u}_1 }_{ \STdata }^2 + \norm{ \underline{u}_1 -  L_{\Delta t} \underline{u}_1 }_{ \Sigma}^2
\leq C \Delta t \sum\limits_{n=1}^{N-1} \left( \norm{ \jump{ \underline{u}_1^n } }_{ \Omega}^2 + \sum\limits_{K \in \mathcal{T}_h } \norm{ \jump{ \underline{u}_1^n }  }_{ L^2(\partial \Omega \cap \partial K )  }^2   \right) \\
& \leq C \Delta t \sum\limits_{n=1}^{N-1} \! \sum\limits_{K \in \mathcal{T}_h } \left( \frac{1}{h} \norm{ \jump{ \underline{u}_1^n }  }_{K }^2 + \! h \norm{ \jump{ \nabla \underline{u}_1^n }  }_{K }^2  \right) \leq C h \abs{ \underline{\mathbf{U}}_h }_{ \uparrow \downarrow }^2 \leq C  h^{2s+1}  \norm{u}_{ H^{m}(\STdom)  }^2,
\end{align*}
where the trace inequality \cref{ieq:cont-trace-ieq} and $\Delta t = C h$  have been used.
Now we have bounded all terms on the right hand side of \cref{eq:appl-stability-lifted-error} and the proof is complete. 
\end{proof}

\section{Preconditioning strategies}\label{section:precond}
The fully-discrete method defined in \cref{eq:opt_compact_fully_disc} is globally coupled in time. This is due to the presence of the term $S^{\uparrow \downarrow}_{\Delta t}( \underline{\mathbf{U}}_h, \underline{\mathbf{W}}_h )$ defined in \cref{eq:def-jump-stab}, which contains jump operators on both variables $\underline{\mathbf{U}}_h$ and $\underline{\mathbf{W}}_h$. Consequently, solving \cref{eq:opt_compact_fully_disc} with a direct solver would require to assemble and factorize a $d+1$-dimensional linear system, which is prohibitively expensive for $d=3$. Hence, iterative solutions strategies must be investigated. \par
In this setting Krylov methods such as GMRes are very attractive as they only require matrix-vector multiplications. Note here that application of the space-time matrix resulting from \cref{eq:opt_compact_fully_disc} can be easily implemented time-slab wise which obliterates the need for assembling an enormous space-time system. Unfortunately, an unpreconditioned GMRes iteration is not suitable for solving our problem even in the simplest possible setting as \cref{fig:pre-1d} demonstrates. The column `NoPre' in the table gives the GMRes iteration numbers for a simple problem on the unit interval discretized with piecewise linear elements, which are almost as high as the total number of degrees of freedom (`ndof' in the second column). Hence, a preconditioning strategy is required.  In this chapter we propose two simple strategies based on time-marching procedures.

\subsection{Monolithic time-marching}\label{ssection:MTM}
A simple preconditioner can be obtained by relaxing the bi-directional coupling $\uparrow \downarrow$ in $S^{\uparrow \downarrow}$ to a forward coupling $\uparrow$ in time only. This is achieved by replacing $S^{\uparrow \downarrow}$ in \cref{eq:complete_bfi_def_fully_discr} by 
\begin{equation}\label{eq:def-jump-stab-pre}
S^{\uparrow }_{\Delta t}( \underline{\mathbf{U}}_h, \underline{\mathbf{W}}_h ) 
:= \underline{I}_1^{\uparrow}( \underline{\mathbf{U}}_h, \underline{\mathbf{W}}_h ) + 
\underline{I}_2^{\uparrow}( \underline{\mathbf{U}}_h, \underline{\mathbf{W}}_h ), 
\end{equation} 
\begin{align*}
\underline{I}_1^{\uparrow}( \underline{\mathbf{U}}_h, \underline{\mathbf{W}}_h ) :=& \sum\limits_{n=1}^{N-1} \left\{ \frac{1}{\Delta t} ( \jump{ \underline{u}_1^n } ,  \underline{w}_{1,+}^n  )_{ \Omega} 
+ \Delta t ( \jump{ \nabla \underline{u}_1^n } , \nabla \underline{w}_{1,+}^n  )_{ \Omega} \right\}, \\ 
 \underline{I}_2^{\uparrow}( \underline{\mathbf{U}}_h, \underline{\mathbf{W}}_h ) :=& \sum\limits_{n=1}^{N-1} \frac{1}{\Delta t} ( \jump{ \underline{u}_2^n } ,  \underline{w}_{2,+}^n  )_{ \Omega}.
\end{align*}
This allows to solve for $(\underline{\mathbf{U}}_h,\underline{\mathbf{Z}}_h)$ time-slab by time-slab, i.e.\ starting with $Q^0$ and progressing one slab $Q^n$ at a time. Indeed, if the solution on $Q^{n-1}$ has been determined then in particular $\underline{u}_{j,-}^n$ for $j=1,2$ are known and can be shifted to the right hand side of the linear system as is common practice in the dG time-stepping method, see e.g.\ the monograph of Thom{\'e}e \cite{Thom07} for details. This means that the cost of one application of the preconditioner is the same as the solution of a wave equation with a dG time-stepping method. \\
Since the dual stabilization and the slab-local part $S_h(\cdot,\cdot)$ of the primal stabilization remain fully active in the preconditioner, it remains possible to choose the polynomial degrees of primal and dual variable independent of each other.
\begin{itemize}
\item In the following we will refer to the strategy using the full dual order, i.e.\ $k_{\ast}=k$ and $ q_{\ast} = q$, as `M-f'. 
\item The choice using minimal dual order $k_{\ast}=1, q_{\ast}=0$ will be denoted by `M-l'.
\end{itemize}

\subsection{Decoupled forward-backward solve}
The strategy proposed in the previous subsection treats primal and dual variable in a monolithic fashion. Alternatively, we can (approximately) solve for these variables separately by exploiting the fact that $\underline{\mathbf{Z}}_h$ converges to zero, see \cref{lem:triple_norm_fully_disc_conv}. Hence, we may set $\underline{\mathbf{Z}}_h = 0$ in \cref{eq:opt_compact_fully_disc} 
and solve for $\underline{\mathbf{U}}_h$ and subsequently determine a corrected  $\underline{\mathbf{Z}}_h$ based on $\underline{\mathbf{U}}_h$. This requires, however, that $A[ \cdot, \cdot]$ on its own provides a consistent and stable discretization of the wave equation which is currently not the case. Consequently, we have to replace $A[ \cdot, \cdot]$ by an enriched bilinear form $\tilde{A}[ \cdot, \cdot]$ defined as follows:
\begin{align}
\tilde{A}[ \underline{\mathbf{U}}_h, \underline{\mathbf{Y}}_h] &=  A[ \underline{\mathbf{U}}_h, \underline{\mathbf{Y}}_h]  \nonumber \\ 
 & + ( \underline{u}_1 , \underline{y}_1 )_{\STdata } + \frac{\lambda}{h} \sum\limits_{n=0}^{N-1} ( \underline{u}_1 , \underline{y}_1  )_{ \Sigma^n } + \sum\limits_{n=1}^{N-1}  \left\{  ( \jump{ \underline{u}_1^n } ,  \underline{y}_{2,+}^n  )_{\Omega} +  ( \jump{ \underline{u}_2^n } ,  \underline{y}_{1,+}^n  )_{\Omega}  \right\} \label{eq:tildeA-def}  
\end{align}
with the Nitsche parameter $\lambda > 0$ sufficiently large as usual. We remark that this enriched bilinear form is a minor modification of a method originally proposed and analyzed in \cite{J93}. In order to account for the presence of the jumps terms in  $\tilde{A}[ \cdot, \cdot]$, the dual stabilization has to be adapted as follows: 
\[
\tilde{S}^{\ast}_h( \underline{\mathbf{Y}}_h, \underline{\mathbf{Z}}_h) = S^{\ast}_h( \underline{\mathbf{Y}}_h, \underline{\mathbf{Z}}_h) + \Delta t \sum\limits_{n=1}^{N-1}  \left\{  ( \underline{y}_{1,+}^n, \underline{z}_{1,+}^n  )_{\Omega} +  ( \underline{y}_{2,+}^n, \underline{z}_{2,+}^n  )_{\Omega}  \right\}.
\]
In $\tilde{A}$ we have also added a Luenberger observer type term on $\STdata$. It is known that this can improve the stability of the corrected dynamics (see \cite{CCDM12,CCM12}), which may improve the convergence of the iterations. Since the perturbation is consistent it is straightforward to include in the error analysis above.
The first-order optimality conditions then become:
Find $(\underline{\mathbf{U}}_h, \underline{\mathbf{Z}}_h) \in \ProdFullyDiscrSpace{k}{q} \times \ProdFullyDiscrSpace{ k_{\ast} }{ q_{\ast} } $ such that
\begin{align}
	(\underline{u}_1,\underline{w}_1)_{\STdata} + \tilde{A}[ \underline{\mathbf{W}}_h, \underline{\mathbf{Z}}_h ] + \tilde{S}_h(  \underline{\mathbf{U}}_h, \underline{\mathbf{W}}_h )  &= (u_{\omega},\underline{w}_1)_{ \STdata}, \quad \forall \; \underline{\mathbf{W}}_h \in \ProdFullyDiscrSpace{k}{q}, \label{eq:optimality_cond_fb_dual} \\ 
\tilde{A}[ \underline{\mathbf{U}}_h, \underline{\mathbf{Y}}_h ] -  \tilde{S}_h^{\ast}( \underline{ \mathbf{Y}}_h, \underline{\mathbf{Z}}_h)  &=  (u_{\omega},\underline{y}_1)_{ \STdata} , \quad \forall \;  \underline{\mathbf{Y}}_h \in \ProdFullyDiscrSpace{ k_{\ast} }{ q_{\ast} },  \label{eq:optimality_cond_fb_primal} 
\end{align} 
for $ \tilde{S}_h(  \underline{\mathbf{U}}_h, \underline{\mathbf{W}}_h ) := S_h(  \underline{\mathbf{U}}_h, \underline{\mathbf{W}}_h ) + S^{\uparrow \downarrow}_{\Delta t}( \underline{\mathbf{U}}_h, \underline{\mathbf{W}}_h ) $. 
Assuming $k=k_{\ast}$ and $q = q_{\ast}$ we propose the following preconditioning strategy, which we will refer to as `DFB'.
\begin{enumerate}
\item[1.] Set  $\underline{\mathbf{Z}}_h = 0$ in \cref{eq:optimality_cond_fb_primal} and solve 
\begin{equation}\label{eq:tildeA-fb-Uh}
\tilde{A}[ \underline{\mathbf{U}}_h, \underline{\mathbf{Y}}_h ]  =  (u_{\omega},\underline{y}_1)_{ \STdata} , \quad \forall \;  \underline{\mathbf{Y}}_h \in \ProdFullyDiscrSpace{ k_{\ast} }{ q_{\ast} }   
\end{equation}
by marching forward in time to obtain $\underline{\mathbf{U}}_h$. 
\item[2.] Then insert $\underline{\mathbf{U}}_h$ in \cref{eq:optimality_cond_fb_dual} and solve 
\begin{equation}\label{eq:tildeA-fb-Zh}
\tilde{A}[ \underline{\mathbf{W}}_h, \underline{\mathbf{Z}}_h ] = (u_{\omega} - \underline{u}_1,\underline{w}_1)_{ \STdata} -  \tilde{S}_h(  \underline{\mathbf{U}}_h, \underline{\mathbf{W}}_h ), \quad \forall \; \underline{\mathbf{W}}_h \in \ProdFullyDiscrSpace{k}{q},
\end{equation}
backward in time to obtain $\underline{\mathbf{Z}}_h$.  
\end{enumerate}
We point out that we integrate by parts in \cref{eq:tildeA-fb-Zh} to write
\begin{align*}
\begin{aligned} 
& \tilde{A}[ \underline{\mathbf{W}}_h, \underline{\mathbf{Z}}_h] = \sum\limits_{n=0}^{N-1} \big\{ - (\underline{w}_2,\partial_t  \underline{z}_1 + \underline{z}_2 )_{Q^n} + a(\underline{w}_1,\underline{z}_1)_{Q^n}  
 - ( \underline{w}_1 , \partial_t \underline{z}_2)_{Q^n}  \\
& - ( \nabla \underline{w}_1 \cdot \mathbf{n}, \underline{z}_1)_{ \Sigma^n } +  \frac{\lambda}{h} ( \underline{w}_1 , \underline{z}_1  )_{ \Sigma^n }  \big\} + ( \underline{w}_1 ,  \underline{z}_1 )_{ \STdata} -  ( \underline{w}_{2,+}^0 , \underline{z}_{1,+}^0  )_{\Omega} -  ( \underline{w}_{1,+}^0 , \underline{z}_{2,+}^0  )_{\Omega} \\ 
&  +  ( \underline{w}_{2,-}^N , \underline{z}_{1,-}^N  )_{\Omega} +  ( \underline{w}_{1,-}^N , \underline{z}_{2,-}^N  )_{\Omega} - \sum\limits_{n=1}^{N-2} \left\{ (  \underline{w}_{2,-}^{n+1} , \jump{ \underline{z}_1^{n+1} }  )_{\Omega} + (  \underline{w}_{1,-}^{n+1} , \jump{ \underline{z}_2^{n+1} }  )_{\Omega}   \right\},
\end{aligned}
\end{align*}
which then indeed allows to obtain $\underline{\mathbf{Z}}_h$ time-slab wise by starting at $n=N-1$ and progressing backwards in time. 
We also mention that this strategy is limited to the setting in which primal and dual polynomial degrees coincide as we need to make sure that the linear systems in \cref{eq:tildeA-fb-Uh} and \cref{eq:tildeA-fb-Zh} are both well-posed.

\subsection{Comparison}
Let us briefly compare the proposed strategies. Both approaches only require ingredients available in standard dG time-stepping methods. In particular, only linear systems on the time slabs $Q^n$ need to be solved. When the full polynomial degree for the dual variables is used, i.e.\  $k_{\ast}=k$ and $ q_{\ast} = q$, then the linear systems on the slabs to be inverted in the decoupled method (`DFB') are half the size of the ones in the monolithic method (`M-f'). This is because of the monolithic treatment of primal and dual variable in the latter. However, if the variant `M-l' of the monolithic method with minimal dual polynomial order is used, then the difference in size between the linear systems quickly diminishes as the primal polynomial order increases. Therefore, it seems fair to focus in the numerical experiments presented in \cref{section:numexp} on a comparison between the `DFB' and `M-l' methods. 
\section{Numerical experiments}\label{section:numexp}
In this section we present some numerical experiments for the fully-discrete method defined in \cref{ssection:fully-discrete-method} and a comparison of the preconditioning strategies proposed in \cref{section:precond}. Our implementation uses the \texttt{FEniCSx} library \cite{BarattaEtal2023,ScroggsEtal2022,BasixJoss,AlnaesEtal2014}. Reproduction material for the numerical experiments is available at Zenodo \cite{zenodo_wave_uc_dg}.

\subsection{Preconditioning experiments for $d=1$ }\label{ssection:precond-1d}
We start by considering a simple experiment on the unit interval $\Omega = [0,1]$. The data domain is given by $\omega = [0,1/4] \cup [3/4,1]$ and the exact solution is $u = \cos(\pi t) \sin(\pi x)$. We measure up to $T=1/2$ so that the GCC is fulfilled. \\
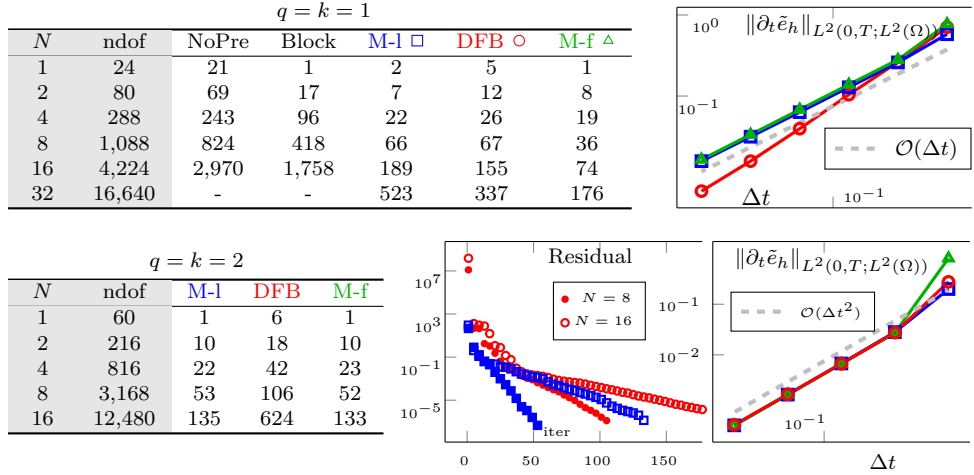
\begin{figure}[h]
\begin{minipage}[c]{0.645\textwidth}%
\begin{footnotesize}
\begin{center}
$q=k=1$
\end{center}
\begin{filecontents*}{precond-1d-iters-order1.dat}
N ndof DFB MTM-full MTM-lo block vanilla
1 24 5 1 2 1 21
2 80 12 8 7 17 69
4 288 26 19 22 96 243
8 1088 67 36 66 418 824
16 4224 155 74 189 1758 2970
32 16640 337 176 523 3000 3000
\end{filecontents*}
\pgfplotstabletypeset[
    columns={N,ndof,vanilla,block,MTM-lo,DFB,MTM-full},
    columns/N/.style={  column name={$N$}, column type/.add={>{ \columncolor[gray]{.9} }}{}  },
    columns/ndof/.style={  column name={ndof}, column type/.add={>{ \columncolor[gray]{.9} }}{}    },
    columns/vanilla/.style={  column name={NoPre}, string replace={3000}{},empty cells with={-},  },
    columns/block/.style={  column name={Block}, string replace={3000}{},empty cells with={-},   },
    columns/MTM-lo/.style={  column name={\textcolor{blue}{M-l} \drawSquare}, string replace={3000}{},empty cells with={-}  },
	columns/DFB/.style={  column name={\textcolor{red}{DFB}\drawCircle}, string replace={3000}{},empty cells with={-}  }, 
    columns/MTM-full/.style={  column name={\textcolor{green!70!black}{M-f} \drawTriangle}, string replace={3000}{},empty cells with={-}  },
    ] {precond-1d-iters-order1.dat}
\end{footnotesize}
\end{minipage}
\begin{minipage}[l]{0.325\textwidth}%
\begin{tikzpicture}[scale = 1.0]
\begin{axis}[
   height = 4.25cm,
   width = 5.5cm,
   every axis plot/.append style={thick},
   axis y line*=left,
   label style={font=\tiny},
   tick label style={font=\tiny}, 
   ymode=log,
   xmode=log,
   ytick={1e0,1e-1},
   xlabel= { \small{ $\Delta t$}  },
   y tick label style={ xshift=2.0em,yshift=-0.4em },
   x tick label style={ xshift=1.2em,yshift=1.1em },
   x label style={at={(axis description cs:0.35,+0.25)},anchor=east}, 
	title = {  \footnotesize{ $\norm{\partial_t \tilde{e}_h }_{L^{2}(0,T;L^2(\Omega))}$}    },
    legend style={at={(0.75, 0.4)},anchor=north},
    title style={at={(0.55,0.935)},anchor=north},
   ]
   \addplot[red,very thick,mark=o,forget plot]
        table[x=deltat,y=DFB] {precond-1d-L2L2ut-order1.dat}; 
   \addplot[blue,very thick,mark=square,forget plot]
        table[x=deltat,y=MTM-lo] {precond-1d-L2L2ut-order1.dat};  
   \addplot[green!70!black,very thick,mark=triangle,forget plot]
        table[x=deltat,y=MTM-full] {precond-1d-L2L2ut-order1.dat};  
   \addplot[lightgray,dashed,ultra thick]
        table[mark=none,x=deltat,y expr ={.75*\thisrowno{0}}] {precond-1d-L2L2ut-order1.dat};  %
	\legend{\footnotesize{ $ \mathcal{O}(\Delta t) $ }}
 \end{axis}
\end{tikzpicture}
\end{minipage}
\vspace*{.75em}
\begin{minipage}[c]{0.385\textwidth}%
\begin{footnotesize}
\begin{center}
$q=k=2$
\end{center}
\begin{filecontents*}{precond-1d-iters-order2.dat}
N ndof DFB MTM-full MTM-lo
1 60 6 1 1
2 216 18 10 10
4 816 42 23 22
8 3168 106 52 53
16 12480 624 133 135
\end{filecontents*}
\pgfplotstabletypeset[
    columns={N,ndof,MTM-lo,DFB,MTM-full},
    columns/N/.style={  column name={$N$}, column type/.add={>{ \columncolor[gray]{.9} }}{}  },
    columns/ndof/.style={  column name={ndof}, column type/.add={>{ \columncolor[gray]{.9} }}{}    },
    columns/MTM-lo/.style={  column name={\textcolor{blue}{M-l}}, string replace={3000}{},empty cells with={-}  },
    columns/DFB/.style={  column name={\textcolor{red}{DFB}}, string replace={3000}{},empty cells with={-}  }, 
    columns/MTM-full/.style={  column name={\textcolor{green!70!black}{M-f}}, string replace={3000}{},empty cells with={-}  },
    ] {precond-1d-iters-order2.dat}
\end{footnotesize}
\end{minipage}
\begin{minipage}[l]{0.615\textwidth}%
 \vspace*{.75em}
\begin{tikzpicture}[scale = 1.0]
\begin{groupplot}[
    group style={
        group size=2 by 1,
        horizontal sep=4.0pt,
        vertical sep=40pt,
   },
   label style={font=\tiny},
   tick label style={font=\tiny}, 
   height = 4.25cm,
   width = 5.0cm,
   every axis plot/.append style={thick},
   axis y line*=left,
   ]
  \nextgroupplot[
  ymode = log,	 
   xlabel= { iter  },
   y tick label style={ xshift=0.25em,yshift=-0.4em },
   x tick label style={ xshift=-0.1em,yshift=-0.2em }, 
   ytick = {1e7,1e3,1e-1,1e-5 },
   x label style={at={(axis description cs:0.525,+0.25)},anchor=east}, 
   title = {\footnotesize{ Residual} },
   legend style={at={(0.6, 0.8)},anchor=north},
   title style={at={(0.55,0.935)},anchor=north},
   xmax = 177,
  ]	
  \addplot[red,only marks, mark =*,mark options={scale=0.5},each nth point=4 ]
        table[x=iter,y=res] {precond-1d-DFB-order2-residuals-ref-lvl3.dat}; 
  \addplot[ red,only marks, mark =o,mark options={scale=0.7},each nth point=4 ]
        table[x=iter,y=res] {precond-1d-DFB-order2-residuals-ref-lvl4.dat}; 

  \addplot[ blue,only marks, mark = square*,mark options={scale=0.7},forget plot,each nth point=4  ]
        table[x=iter,y=res] {precond-1d-MTM-lo-order2-residuals-ref-lvl3.dat}; 
  \addplot[ blue,only marks, mark =square,mark options={scale=0.7},forget plot,each nth point=4  ]
        table[x=iter,y=res] {precond-1d-MTM-lo-order2-residuals-ref-lvl4.dat}; 
	\legend{  \tiny{ $N=8$}, \tiny{$N=16$} }
 
  \nextgroupplot[ 
   ymode=log,
   xmode=log,
   xtick = {1e-1},
   ytick={1e-1,1e-2},
   axis y line*=left,
   xlabel= { \footnotesize{ $\Delta t$}  },
   y tick label style={ xshift=-0.1em,yshift= 0.0em },
   x tick label style={ xshift=-0.7em,yshift=1.4em },
   x label style={at={(axis description cs:0.75,0.1)},anchor=east}, 
	title = {  \footnotesize{ $\norm{\partial_t \tilde{e}_h }_{L^{2}(0,T;L^2(\Omega))}$}    },
    legend style={at={(0.35, 0.75)},anchor=north},
    title style={at={(0.45,0.935)},anchor=north},
  ]
   \addplot[blue,very thick,mark=square,forget plot]
        table[x=deltat,y=MTM-lo] {precond-1d-L2L2ut-order2.dat};  
   \addplot[green!70!black,very thick,mark=triangle,forget plot]
        table[x=deltat,y=MTM-full] {precond-1d-L2L2ut-order2.dat};  
   \addplot[red,very thick,mark=o,forget plot]
        table[x=deltat,y=DFB] {precond-1d-L2L2ut-order2.dat}; 
   \addplot[lightgray,dashed,ultra thick]
        table[mark=none,x=deltat,y expr ={.75*\thisrowno{0}*\thisrowno{0}}] {precond-1d-L2L2ut-order2.dat};  %
	\legend{\tiny{ $ \mathcal{O}(\Delta t^2) $ }}
\end{groupplot}
\end{tikzpicture}

\end{minipage}
\caption{ Results for one-dimensional experiment of \cref{ssection:precond-1d} using $h = \Delta t$. 
The upper panel refers to $q=k=1$ while the lower panel uses $q=k=2$. The tables contain iteration numbers for GMRes preconditioned by the methods proposed in \cref{section:precond} using tolerance $10^{-7}$. The column `ndof' gives the total number of degrees of freedom for the `M-f'-method.
The two plots on the right show convergence of $\norm{\partial_t \tilde{e}_h }_{L^{2}(0,T;L^2(\Omega))}$ with $\tilde{e}_h := u - L_{\Delta t} \underline{u}_1$. In the history of the residual only every fourth iteration is shown as a data point.  
}
\label{fig:pre-1d}
\end{figure}
The results for piecewise affine and quadratic approximation are displayed in \cref{fig:pre-1d}.
In addition to the preconditioning approaches introduced in \cref{section:precond} we also show the iteration numbers for a `Block'-Jacobi preconditioner in which the coupling terms $S^{\uparrow \downarrow}_{\Delta t}( \cdot, \cdot )$ have been completely removed. The poor performance of this approach shows that including some form of communication between time-slabs is crucial. Concerning the approaches from  \cref{section:precond} we make the follwing observations:
\begin{itemize}
\item The iteration numbers for all the methods increase at least linearly with $N$. For the monolithic method this is evident as the term $S^{\uparrow \downarrow}_{\Delta t}(\cdot,\cdot)$ which is only approximately represented by the preconditioner scales with $1/\Delta t \sim N$. 
\item The method `M-f' has the lowest iteration numbers. However, it is also the most expensive. 
Fortunately, the iteration numbers for the `M-l' method are nearly the same as for the  `M-f' method when $q=k >1$ and its cost is similar to that of the `DFB' method. Hence, the `M-l' method is a promising candidate to be used for higher-order discretizations.  
\item For $q=k = 1$ the  `M-l' method cannot compete with the  `DFB' method due to the poorly resolved dual variable. The plot of the error in the upper panel of \cref{fig:pre-1d} also shows that the latter enjoys some super-convergence properties for piecewise affine discretizations. This is courtesy of the additional jumps terms which have been integrated into the modified bilinear form $\tilde{A}[ \cdot, \cdot]$ in \cref{eq:tildeA-def}. These terms cannot be accommodated in the method `M-l' as they are incompatable with the monolithic forward sweep of \cref{ssection:MTM}. 
A minor drawback of the `DFB' method is that the corresponding residuals are very high in the first GMRes iterations and they also increase with $N$ as \cref{fig:pre-1d} shows. Here, the  `M-l' method behaves more favourably. Nevertheless, the `DFB' method is attractive for  $q=k = 1$. 
\end{itemize}

\subsection{Unit cube geometry with data given all around}\label{ssection:cube-gcc}
Let us now consider a computationally more demanding experiment on the unit cube $\Omega = [0,1]^3$ with data domain $\omega = \Omega \setminus [1/4,3/4]^3$ and final time $T= 1/2$.  The data domain is sketched in the center of \cref{fig:GCC-3D}. The reference solution is given by 
\[
u(t,x,y,z) = \cos(\sqrt{3} \pi t) \sin(\pi x) \sin(\pi y) \sin(\pi z).  
\]
We set $\Delta t = h$ for this experiment and report the errors in the $\norm{ \cdot }_{L^{\infty}(0,T;L^2(\Omega))}$ and $\norm{ \partial_t( \cdot) }_{L^{2}(0,T;L^2(\Omega))}$-norms\footnote{For simplicity we measured the $L^{2}(0,T;L^2(\Omega))$ instead of the $L^{2}(0,T;H^{-1}(\Omega))$ norm.} in terms of the polynomial order $q=k$ in \cref{fig:GCC-3D}.  For the wave displacement some superconvergence can be observed, as has also been reported in previous publications \cite{BFO20,BFMO21} for the spatially one-dimensional case. The velocity marginally superconverges for the 
`DFB' method when piecewise affine approximation is used as already observed in the previous experiment. 
Overall, the numerical convergence rates are in line with the theoretical result from \cref{thm:error-estimate-fully-discrete}. \\ 
\begin{figure}[h]
\centering
\begin{tikzpicture}[scale = 1.0]
\begin{scope}[ ]
\begin{tiny}
\begin{axis}[
    height = 4.75cm,
    width = 5.7cm,
    name=displ,
    ymode=log,
    xmode=log,
    ymax = 2e-1,
    axis y line*=left,
    xlabel= { $\Delta t = h$},
    x label style={at={(axis description cs:0.35,+0.225)},anchor=east},
    legend style = { column sep = 10pt, legend columns = 1, legend to name = grouplegend,},
    title style={at={(0.5,1.0735)},anchor=north},
    title = {  $\norm{ u - \mathcal{L}_{\Delta t} \underline{u}_1 }_{L^{\infty}(0,T;L^2(\Omega))}$ },
    legend style={at={(0.5,-0.1)},anchor=north},
	]
    
    \addplot[blue,only marks,mark=square,mark options={scale=1.25},forget plot] 
   	table[x=deltat,y=MTM-lo] {precond-3d-GCC-LinftyL2u-order1.dat}; 
    \addplot[blue,very thick] 
   	table[x=deltat,y=MTM-lo] {precond-3d-GCC-LinftyL2u-order1.dat}; \addlegendentry{$q=1$}%
     \addplot[red,very thick,only marks, mark=o,mark options={scale=1.25},forget plot] 
   	table[x=deltat,y=DFB] {precond-3d-GCC-LinftyL2u-order1.dat}; 
    
    \addplot[blue,only marks,mark=square,mark options={scale=1.25},forget plot] 
   	table[x=deltat,y=MTM-lo] {precond-3d-GCC-LinftyL2u-order2.dat}; 
    \addplot[blue,very thick,dashed] 
   	table[x=deltat,y=MTM-lo] {precond-3d-GCC-LinftyL2u-order2.dat}; \addlegendentry{$q=2$}%
     \addplot[red,very thick,only marks, mark=o,mark options={scale=1.25},forget plot] 
   	table[x=deltat,y=DFB] {precond-3d-GCC-LinftyL2u-order2.dat}; 

    \addplot[blue,only marks,mark=square,mark options={scale=1.25},forget plot] 
   	table[x=deltat,y=MTM-lo] {precond-3d-GCC-LinftyL2u-order3.dat}; 
    \addplot[blue,very thick,densely dotted] 
   	table[x=deltat,y=MTM-lo] {precond-3d-GCC-LinftyL2u-order3.dat}; \addlegendentry{$q=3$}%
    \addplot[red,very thick,only marks, mark=o,mark options={scale=1.25},forget plot] 
   	table[x=deltat,y=DFB] {precond-3d-GCC-LinftyL2u-order3.dat}; 

    \addplot[lightgray,dashed,ultra thick] 
    	table[mark=none,x=deltat,y expr ={.15*\thisrowno{0}}] {precond-3d-GCC-LinftyL2u-order1.dat};  
    \addplot[lightgray,dotted,ultra thick] 
    	table[mark=none,x=deltat,y expr ={.15*\thisrowno{0}*\thisrowno{0} }] {precond-3d-GCC-LinftyL2u-order2.dat};  
    \addplot[lightgray,dashdotted,ultra thick] 
    	table[mark=none,x=deltat,y expr ={.05*\thisrowno{0}*\thisrowno{0}*\thisrowno{0}}] {precond-3d-GCC-LinftyL2u-order3.dat};  
   \end{axis}
    \node at ($(displ) + (-0.0cm,-2.55cm)$) {\ref{grouplegend}}; 
\end{tiny}
\end{scope}

 \begin{scope}[xshift=4.5cm]
 \node (geom) at (1.0,2.25) {\includegraphics[scale =.15]{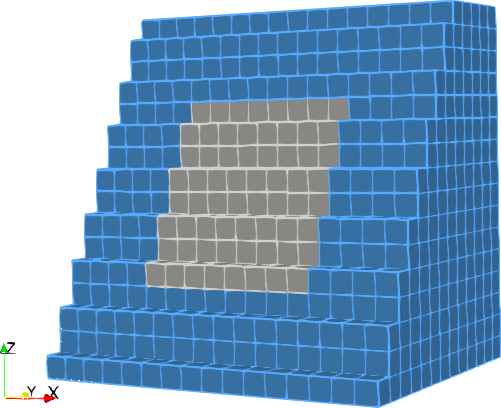}};
 \node[draw,fill=white] (la) at (1.15,3.25) { {\footnotesize $\omega $ }};
  \node[] (Z2) at (-2.25,0.55) {};
  \node (err) at (1.25,-0.4) {\includegraphics[scale =.115]{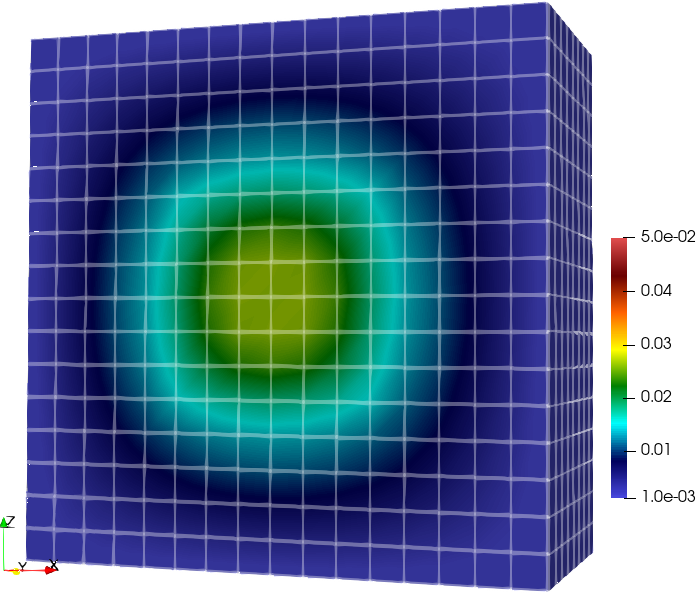}};
 \node[draw,fill=white] (la) at (1.05,0.875) {\tiny{$\vert u(0,\cdot) - \mathcal{L}_{\Delta t} \underline{u}_1(0,\cdot) \vert$}};
  \draw[lightgray,ultra thick, ->] (-0.3,-0.55 ) -- (-3.145, 1.695);
 \end{scope}

\begin{scope}[xshift=7.5cm]
\begin{tiny}
\begin{axis}[ 
    height = 4.75cm,
    width = 5.7cm,
    name=vel,
    ymode=log,
    xmode=log,
    ymax = 1e-0,
    axis y line*=left,
    xlabel= { $\Delta t = h$},
    legend style = { column sep = 10pt, legend columns = 1, legend to name = grouplegendv,},
    x label style={at={(axis description cs:0.35,+0.225)},anchor=east},
    title style={at={(0.5,1.0735)},anchor=north},
    title = {  $\norm{ \partial_t( u - \mathcal{L}_{\Delta t} \underline{u}_1 ) }_{L^{2}(0,T;L^2(\Omega))}$ },
    legend style={at={(0.5,-0.1)},anchor=north},
	]
    
    \addplot[blue,only marks,mark=square,mark options={scale=1.25},forget plot] 
   	table[x=deltat,y=MTM-lo] {precond-3d-GCC-L2L2ut-order1.dat}; 
    \addplot[blue,very thick,forget plot] 
   	table[x=deltat,y=MTM-lo] {precond-3d-GCC-L2L2ut-order1.dat}; 
     \addplot[red,very thick,only marks, mark=o,mark options={scale=1.25},forget plot] 
   	table[x=deltat,y=DFB] {precond-3d-GCC-L2L2ut-order1.dat}; 
    
    \addplot[blue,only marks,mark=square,mark options={scale=1.25},forget plot] 
   	table[x=deltat,y=MTM-lo] {precond-3d-GCC-L2L2ut-order2.dat}; 
    \addplot[blue,very thick,dashed,forget plot] 
   	table[x=deltat,y=MTM-lo] {precond-3d-GCC-L2L2ut-order2.dat}; 
     \addplot[red,very thick,only marks, mark=o,mark options={scale=1.25},forget plot] 
   	table[x=deltat,y=DFB] {precond-3d-GCC-L2L2ut-order2.dat}; 

    \addplot[blue,only marks,mark=square,mark options={scale=1.25},forget plot] 
   	table[x=deltat,y=MTM-lo] {precond-3d-GCC-L2L2ut-order3.dat}; 
    \addplot[blue,very thick,densely dotted,forget plot] 
   	table[x=deltat,y=MTM-lo] {precond-3d-GCC-L2L2ut-order3.dat}; 
    \addplot[red,very thick,only marks, mark=o,mark options={scale=1.25},forget plot] 
   	table[x=deltat,y=DFB] {precond-3d-GCC-L2L2ut-order3.dat}; 
 
    \addplot[lightgray,dashed,ultra thick] 
    	table[mark=none,x=deltat,y expr ={1.3*\thisrowno{0}}] {precond-3d-GCC-L2L2ut-order1.dat};  \addlegendentry{$ \mathcal{O}(\Delta t) $ } %
    \addplot[lightgray,dotted,ultra thick] 
    	table[mark=none,x=deltat,y expr ={1.4*\thisrowno{0}*\thisrowno{0} }] {precond-3d-GCC-L2L2ut-order2.dat};  \addlegendentry{$ \mathcal{O}((\Delta t)^2) $ } %
    \addplot[lightgray,dashdotted,ultra thick] 
    	table[mark=none,x=deltat,y expr ={.5*\thisrowno{0}*\thisrowno{0}*\thisrowno{0}}] {precond-3d-GCC-L2L2ut-order3.dat};  \addlegendentry{$ \mathcal{O}((\Delta t)^3) $ } %
   \end{axis}
    \node at ($(vel) + (-0.0cm,-2.55cm)$) {\ref{grouplegendv}}; 
\end{tiny}
\end{scope}
\end{tikzpicture}
\caption{Convergence plots for the unit cube with data given in  $\omega = \Omega \setminus [1/4,3/4]^3$. Results are given for the `M-l' {\protect \drawSquare} and `DFB'{\protect \drawCircle} methods. The central figure displays the absolute error on the indicated refinement level. }
\label{fig:GCC-3D}
\end{figure}
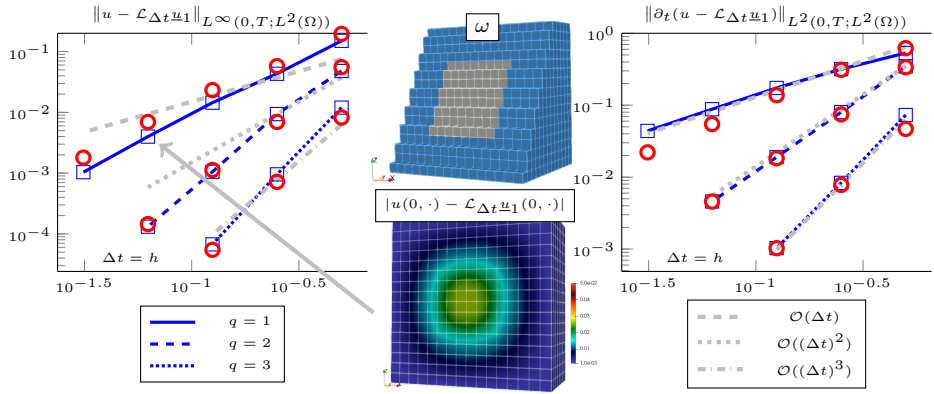

\begin{table}[h]
\begin{center}
$q=k=1$ \\
\begin{filecontents*}{precond-3d-GCC-iters-order1.dat}
N ndof-tot-DFB ndof-inv-DFB ndof-tot-MTM-lo ndof-inv-MTM-lo DFB-iter MTM-lo-iter
1 216 108 162 162 5 2
2 2000 500 1500 750 16 9
4 23328 2916 17496 4374 45 44
8 314432 19652 235824 29478 106 209
16 4599936 143748 3449952 215622 225 856
\end{filecontents*}
\pgfplotstabletypeset[ 
    every head row/.style={
    before row={
    \bottomrule
    	& \multicolumn{3}{c}{\textcolor{red}{DFB}} & \multicolumn{3}{c}{\textcolor{blue}{M-l}}\\
	},
	after row=\toprule,
    },
    columns={N,ndof-tot-DFB,ndof-inv-DFB,DFB-iter,ndof-tot-MTM-lo,ndof-inv-MTM-lo,MTM-lo-iter},
    columns/N/.style={  column name={$N$}, column type/.add={>{ \columncolor[gray]{.9} }}{}  },
    columns/ndof-tot-DFB/.style={  column name={ndof-tot}    },
    columns/ndof-inv-DFB/.style={  column name={ndof-inv}    },
    columns/DFB-iter/.style={  column name={iter} },
    columns/ndof-tot-MTM-lo/.style={  column name={ndof-tot}    },
    columns/ndof-inv-MTM-lo/.style={  column name={ndof-inv}    },
    columns/MTM-lo-iter/.style={  column name={iter} },
    ] {precond-3d-GCC-iters-order1.dat}
\\
$q=k=2$ \\
\begin{filecontents*}{precond-3d-GCC-iters-order2.dat}
N ndof-tot-DFB ndof-inv-DFB ndof-tot-MTM-lo ndof-inv-MTM-lo DFB-iter MTM-lo-iter
1 1500 750 804 804 14 2
2 17496 4374 9248 4624 38 11
4 235824 29478 123744 30936 105 25
8 3449952 215622 1803584 225448 239 81
\end{filecontents*}
\pgfplotstabletypeset[ 
    every head row/.style={
    before row={
    \bottomrule
    	& \multicolumn{3}{c}{\textcolor{red}{DFB}} & \multicolumn{3}{c}{\textcolor{blue}{M-l}}\\
	},
	after row=\toprule,
    },
    columns={N,ndof-tot-DFB,ndof-inv-DFB,DFB-iter,ndof-tot-MTM-lo,ndof-inv-MTM-lo,MTM-lo-iter},
    columns/N/.style={  column name={$N$}, column type/.add={>{ \columncolor[gray]{.9} }}{}  },
    columns/ndof-tot-DFB/.style={  column name={ndof-tot}    },
    columns/ndof-inv-DFB/.style={  column name={ndof-inv}    },
    columns/DFB-iter/.style={  column name={iter} },
    columns/ndof-tot-MTM-lo/.style={  column name={ndof-tot}    },
    columns/ndof-inv-MTM-lo/.style={  column name={ndof-inv}    },
    columns/MTM-lo-iter/.style={  column name={iter} },
    ] {precond-3d-GCC-iters-order2.dat}
\\
$q=k=3$ \\
\begin{filecontents*}{precond-3d-GCC-iters-order3.dat}
N ndof-tot-DFB ndof-inv-DFB ndof-tot-MTM-lo ndof-inv-MTM-lo DFB-iter MTM-lo-iter
1 5488 2744 2798 2798 17 2
2 70304 17576 35652 17826 79 12
4 1000000 125000 505832 126458 223 28
\end{filecontents*}
\pgfplotstabletypeset[ 
    every head row/.style={
    before row={
    \bottomrule
    	& \multicolumn{3}{c}{\textcolor{red}{DFB}} & \multicolumn{3}{c}{\textcolor{blue}{M-l}}\\
	},
	after row=\toprule,
    },
    columns={N,ndof-tot-DFB,ndof-inv-DFB,DFB-iter,ndof-tot-MTM-lo,ndof-inv-MTM-lo,MTM-lo-iter},
    columns/N/.style={  column name={$N$}, column type/.add={>{ \columncolor[gray]{.9} }}{}  },
    columns/ndof-tot-DFB/.style={  column name={ndof-tot}    },
    columns/ndof-inv-DFB/.style={  column name={ndof-inv}    },
    columns/DFB-iter/.style={  column name={iter} },
    columns/ndof-tot-MTM-lo/.style={  column name={ndof-tot}    },
    columns/ndof-inv-MTM-lo/.style={  column name={ndof-inv}    },
    columns/MTM-lo-iter/.style={  column name={iter} },
    ] {precond-3d-GCC-iters-order3.dat}
\end{center}
\caption{GMRes iteration numbers for the unit cube with data given in  $\omega = \Omega \setminus [1/4,3/4]^3$ using tolerance $10^{-5}$. The left half of the tables displays the results for the `DFB' method and the right half for the `M-l' method. The column `ndof-tot' gives the total number of degrees of freedom of the entire space-time system for the respective method, whereas the column `ndof-inv' indicates the size of the linear system which needs to be inverted to apply the preconditioner.}
\label{tab:iter-3D-GCC}
\end{table}

The GMRes iteration numbers displayed in \cref{tab:iter-3D-GCC} confirm our observations from the one-dimensional example: The `DFB' method is preferred for  $q=k = 1$, whereas the iteration numbers for the `M-l' method are far lower for $q=k > 1$. Since the degrees of freedom in the linear system which needs to be inverted to apply the preconditioner (column `ndof-inv' in the table) are only marginally larger for the `M-l' method, this is clearly the method of choice for higher order discretizations.  
Even though this preconditioning strategy is not robust with respect to refinement of the discretization, it still allows us to solve a globally-coupled unique continuation problem posed on a four dimensional space-time geometry up to a reasonable level of accuracy. Solving the $d+1$ dimensional space-time systems directly would require to invert a linear system which involves a factor\footnote{The ratio of the columns `ndof-tot' and  `ndof-inv' in \cref{tab:iter-3D-GCC}.} $N$ (respectively $2N$ for the `DFB' method) more degrees of freedom than the linear system to be factorized to apply the preconditioner.

\subsection{Unit cube geometry without GCC}

\begin{figure}[h]
\begin{center}
\begin{tikzpicture}[scale = 1.0]
\begin{scope}[ ]
\begin{tiny}
\begin{axis}[
    height = 4.75cm,
    width = 5.7cm,
    name=displ,
    ymode=log,
    xmode=log,
    ymax = 7e-1,
    axis y line*=left,
    xlabel= { $\Delta t = h$},
    x label style={at={(axis description cs:0.55,+0.21)},anchor=east},
    legend style = { column sep = 10pt, legend columns = 1, legend to name = grouplegendn,},
    title = {  $\norm{ u - \mathcal{L}_{\Delta t} \underline{u}_1 }_{L^{\infty}(0,T;L^2(\Omega))}$ }, 
    title style={at={(0.5,1.0735)},anchor=north},
     legend style={at={(0.5,-0.1)},anchor=north},
	]
    
    \addplot[blue,only marks,mark=square,mark options={scale=1.25},forget plot] 
   	table[x=deltat,y=MTM-lo] {precond-3d-noGCC-LinftyL2u-order1.dat}; 
    \addplot[blue,very thick] 
   	table[x=deltat,y=MTM-lo] {precond-3d-noGCC-LinftyL2u-order1.dat}; \addlegendentry{$q=1$}%
    \addplot[red,very thick,forget plot] 
   	table[x=deltat,y=DFB] {precond-3d-noGCC-LinftyL2u-order1.dat};    
     \addplot[red,very thick,only marks, mark=o,mark options={scale=1.25},forget plot] 
   	table[x=deltat,y=DFB] {precond-3d-noGCC-LinftyL2u-order1.dat}; 
    
    \addplot[blue,only marks,mark=square,mark options={scale=1.25},forget plot] 
   	table[x=deltat,y=MTM-lo] {precond-3d-noGCC-LinftyL2u-order2.dat}; 
    \addplot[blue,very thick,dashed] 
   	table[x=deltat,y=MTM-lo] {precond-3d-noGCC-LinftyL2u-order2.dat}; \addlegendentry{$q=2$}%
     \addplot[red,very thick,only marks, mark=o,mark options={scale=1.25},forget plot] 
   	table[x=deltat,y=DFB] {precond-3d-noGCC-LinftyL2u-order2.dat}; 
    \addplot[red,very thick,dashed,forget plot] 
   	table[x=deltat,y=DFB] {precond-3d-noGCC-LinftyL2u-order2.dat};   

    \addplot[blue,only marks,mark=square,mark options={scale=1.25},forget plot] 
   	table[x=deltat,y=MTM-lo] {precond-3d-noGCC-LinftyL2u-order3.dat}; 
    \addplot[blue,very thick,densely dotted] 
   	table[x=deltat,y=MTM-lo] {precond-3d-noGCC-LinftyL2u-order3.dat}; \addlegendentry{$q=3$}%
    \addplot[red,very thick,only marks, mark=o,mark options={scale=1.25},forget plot] 
   	table[x=deltat,y=DFB] {precond-3d-noGCC-LinftyL2u-order3.dat}; 
    \addplot[red,very thick,densely dotted] 
	table[x=deltat,y=DFB] {precond-3d-noGCC-LinftyL2u-order3.dat}; 

    \addplot[lightgray,dashed,ultra thick] 
    	table[mark=none,x=deltat,y expr ={.9*\thisrowno{0}}] {precond-3d-noGCC-LinftyL2u-order1.dat};  
    \addplot[lightgray,dotted,ultra thick] 
    	table[mark=none,x=deltat,y expr ={.35*\thisrowno{0}*\thisrowno{0} }] {precond-3d-noGCC-LinftyL2u-order2.dat};  
    \addplot[lightgray,dashdotted,ultra thick] 
    	table[mark=none,x=deltat,y expr ={.25*\thisrowno{0}*\thisrowno{0}*\thisrowno{0}}] {precond-3d-noGCC-LinftyL2u-order3.dat};  
   \end{axis}
    \node at ($(displ) + (-0.0cm,-2.55cm)$) {\ref{grouplegendn}}; 
\end{tiny}
\end{scope}

 \begin{scope}[xshift=4.5cm]
 \node (geom) at (1.15,2.25) {\includegraphics[scale =.13]{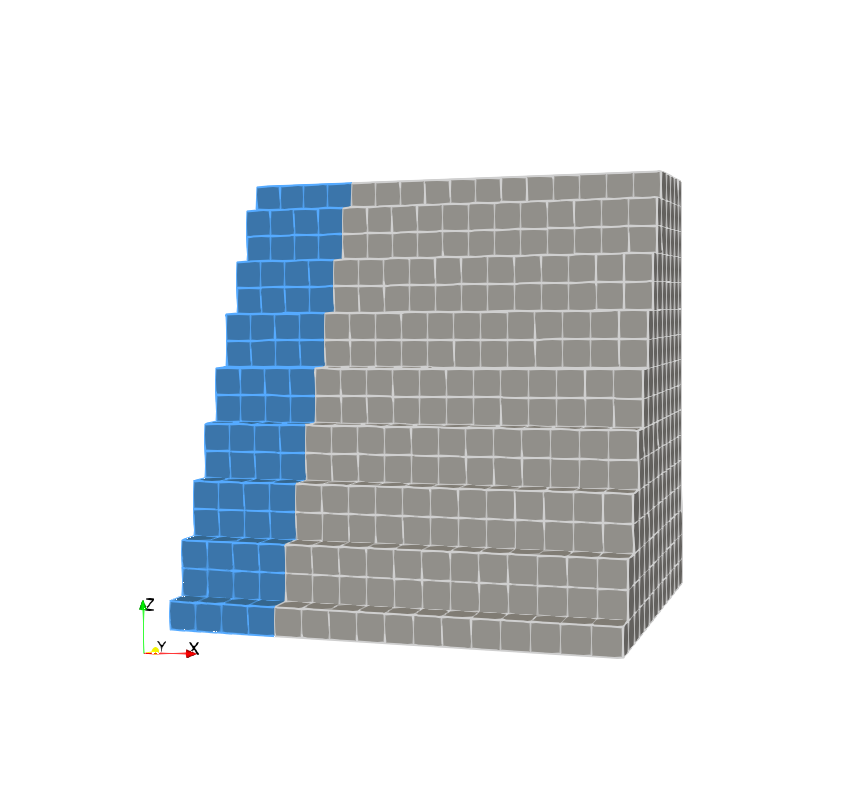}};
 \node[draw,fill=white] (la) at (0.46,3.15) { {\footnotesize $\omega $ }};
  \node[] (Z2) at (-2.25,0.55) {};
  \node (err) at (1.0,-0.4) {\includegraphics[scale =.13]{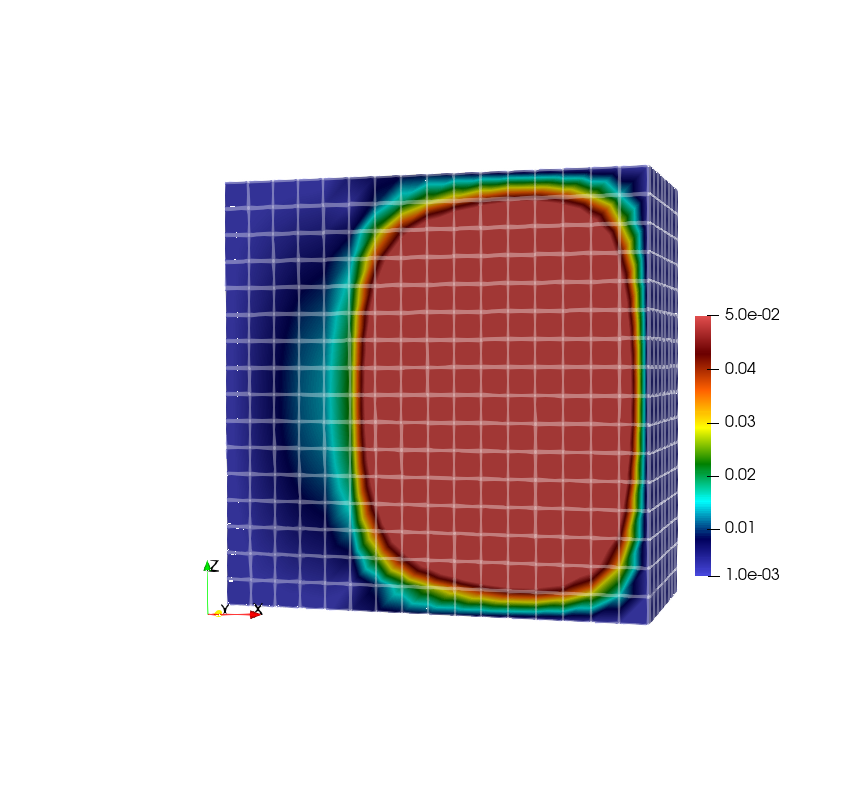}};
 \node[draw,fill=white] (la) at (1.05,0.78) {\tiny{$\vert u(0,\cdot) - \mathcal{L}_{\Delta t} \underline{u}_1(0,\cdot) \vert$}};
 \draw[lightgray,ultra thick, ->] (-0.3,-0.55 ) -- (-3.145, 2.195);
 \end{scope}

\begin{scope}[xshift=7.5cm]
\begin{tiny}
\begin{axis}[ 
    height = 4.75cm,
    width = 5.7cm,
    name=vel,
    ymode=log,
    xmode=log,
    ymax = 1e-0,
    xlabel= { $\Delta t = h$},
    axis y line*=left,
    legend style = { column sep = 10pt, legend columns = 1, legend to name = grouplegendnv,},
    x label style={at={(axis description cs:0.55,+0.21)},anchor=east},
    title = {  $\norm{ \partial_t( u - \mathcal{L}_{\Delta t} \underline{u}_1 ) }_{L^{2}(0,T;L^2(\Omega))}$ },
    legend style={at={(0.5,-0.1)},anchor=north}, 
    title style={at={(0.5,1.0735)},anchor=north},
	]
    
    \addplot[lightgray,dashed,ultra thick] 
    	table[mark=none,x=deltat,y expr ={1.3*\thisrowno{0}}] {precond-3d-noGCC-L2L2ut-order1.dat};  \addlegendentry{$ \mathcal{O}(\Delta t) $ } %
    \addplot[lightgray,dotted,ultra thick] 
    	table[mark=none,x=deltat,y expr ={1.4*\thisrowno{0}*\thisrowno{0} }] {precond-3d-noGCC-L2L2ut-order2.dat};  \addlegendentry{$ \mathcal{O}((\Delta t)^2) $ } %
    \addplot[lightgray,dashdotted,ultra thick] 
    	table[mark=none,x=deltat,y expr ={.5*\thisrowno{0}*\thisrowno{0}*\thisrowno{0}}] {precond-3d-noGCC-L2L2ut-order3.dat};  \addlegendentry{$ \mathcal{O}((\Delta t)^3) $ } %
    
    \addplot[blue,only marks,mark=square,mark options={scale=1.25},forget plot] 
   	table[x=deltat,y=MTM-lo] {precond-3d-noGCC-L2L2ut-order1.dat}; 
    \addplot[blue,very thick,forget plot] 
   	table[x=deltat,y=MTM-lo] {precond-3d-noGCC-L2L2ut-order1.dat}; \addlegendentry{$q=1$}%
    \addplot[red,very thick,forget plot] 
   	table[x=deltat,y=DFB,forget plot] {precond-3d-noGCC-L2L2ut-order1.dat};    
     \addplot[red,very thick,only marks, mark=o,mark options={scale=1.25},forget plot] 
   	table[x=deltat,y=DFB] {precond-3d-noGCC-L2L2ut-order1.dat}; 
    
    \addplot[blue,only marks,mark=square,mark options={scale=1.25},forget plot] 
   	table[x=deltat,y=MTM-lo] {precond-3d-noGCC-L2L2ut-order2.dat}; 
    \addplot[blue,very thick,dashed,forget plot] 
   	table[x=deltat,y=MTM-lo] {precond-3d-noGCC-L2L2ut-order2.dat}; \addlegendentry{$q=2$}%
     \addplot[red,very thick,only marks, mark=o,mark options={scale=1.25},forget plot] 
   	table[x=deltat,y=DFB] {precond-3d-noGCC-L2L2ut-order2.dat}; 
    \addplot[red,very thick,dashed,forget plot] 
   	table[x=deltat,y=DFB] {precond-3d-noGCC-L2L2ut-order2.dat};   

    \addplot[blue,only marks,mark=square,mark options={scale=1.25},forget plot] 
   	table[x=deltat,y=MTM-lo] {precond-3d-noGCC-L2L2ut-order3.dat}; 
    \addplot[blue,very thick,densely dotted,forget plot] 
   	table[x=deltat,y=MTM-lo] {precond-3d-noGCC-L2L2ut-order3.dat}; \addlegendentry{$q=3$}%
    \addplot[red,very thick,only marks, mark=o,mark options={scale=1.25},forget plot] 
   	table[x=deltat,y=DFB] {precond-3d-noGCC-L2L2ut-order3.dat}; 
    \addplot[red,very thick,densely dotted,forget plot] 
	table[x=deltat,y=DFB] {precond-3d-noGCC-L2L2ut-order3.dat}; 
 
   \end{axis}
    \node at ($(vel) + (-0.0cm,-2.55cm)$) {\ref{grouplegendnv}}; 
\end{tiny}
\end{scope}
\end{tikzpicture}
\end{center}
\vspace{-2.0em}
\begin{center}
$q=k=2$ \\
\begin{filecontents*}{precond-3d-noGCC-iters-order2.dat}
N ndof-tot-DFB ndof-inv-DFB ndof-tot-MTM-lo ndof-inv-MTM-lo DFB-iter MTM-lo-iter
1 1500 750 804 804 8 1
2 17496 4374 9248 4624 31 19
4 235824 29478 123744 30936 81 57
8 3449952 215622 1803584 225448 187 211
\end{filecontents*}
\pgfplotstabletypeset[ 
    every head row/.style={
    before row={
    \bottomrule
    	& \multicolumn{3}{c}{\textcolor{red}{DFB}} & \multicolumn{3}{c}{\textcolor{blue}{M-l}}\\
	},
	after row=\toprule,
    },
    columns={N,ndof-tot-DFB,ndof-inv-DFB,DFB-iter,ndof-tot-MTM-lo,ndof-inv-MTM-lo,MTM-lo-iter},
    columns/N/.style={  column name={$N$}, column type/.add={>{ \columncolor[gray]{.9} }}{}  },
    columns/ndof-tot-DFB/.style={  column name={ndof-tot}    },
    columns/ndof-inv-DFB/.style={  column name={ndof-inv}    },
    columns/DFB-iter/.style={  column name={iter} },
    columns/ndof-tot-MTM-lo/.style={  column name={ndof-tot}    },
    columns/ndof-inv-MTM-lo/.style={  column name={ndof-inv}    },
    columns/MTM-lo-iter/.style={  column name={iter} },
    ] {precond-3d-noGCC-iters-order2.dat}
\end{center}

\caption{On top: Convergence plots for the unit cube with data given in merely in $\omega = [0,1/4] \times [0,1]^2$. Results are given for the `M-l' {\protect \drawSquare} and `DFB'{\protect \drawCircle} methods. Bottom: GMRes iteration numbers for $q=k=2$.}
\label{fig:NoGCC-3D}
\end{figure}
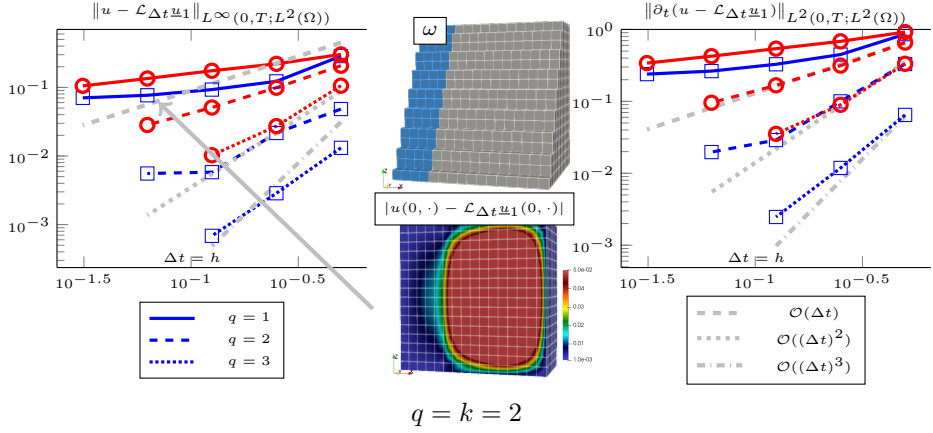
Let us change the setup from \cref{ssection:cube-gcc} by removing part of the data. If we set 
$\omega = [0,1/4] \times [0,1]^2$ keeping the final time $T=1/2$, then the GCC, which is crucial for our analysis, will be violated. 
The numerical results shown in \cref{fig:NoGCC-3D} indicate that the error in this case can only be reduced to a certain threshold (depending on the type of discretization) at which stagnation occurs. Interestingly, the accuracies achieved with the `DFB' and `M-l' methods differ significantly. For $q=k \in \{2,3\}$ the latter is about one order of magnitude more accurate than the former. On the other hand, the GMRes iteration numbers for the `M-l' method are more adversely affected by the lack of the GCC than for the `DFM' method.   
In the end, both methods are hampered by the fact that the error for the reconstruction of the wave displacement outside the data domain is fairly large and will not converge to zero (at a reasonable speed) when the discretization is refined as the central plot in \cref{fig:NoGCC-3D} shows. \\
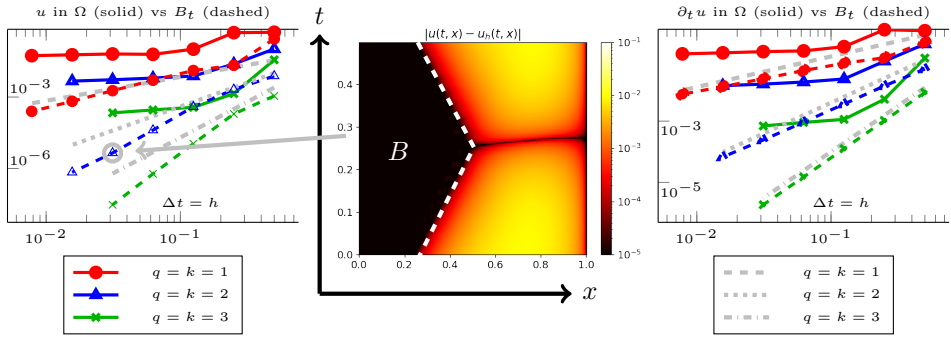
\begin{figure}[h]
\begin{center}
\begin{tikzpicture}[scale = 1.0]
\begin{scope}[ ]
\begin{tiny}
\begin{axis}[
    height = 4.15cm,
    width = 5.45cm,
    name=displ,
    ymode=log,
    xmode=log,
    ymax = 7e-1,
    axis y line*=left,
    xlabel= { $\Delta t = h$},
    x label style={at={(axis description cs:0.75,+0.31)},anchor=east},
    y tick label style={ xshift=2.85em,yshift=0.6em },
    ytick = {1e-3,1e-6},
    legend style = { column sep = 10pt, legend columns = 1, legend to name = grouplegendnoGCC,},
    title = { $ u$ in $ \Omega$ (solid) vs $ B_t $ (dashed) }, 
    title style={at={(0.5,1.0735)},anchor=north},
    legend style={at={(0.5,-0.0)},anchor=north},
    ] 
    
    \addplot[red,very thick,mark=*]
        table[x=deltat,y=LinftyL2u-all] {noGCC-restricted-1d-order1.dat}; \addlegendentry{$q=k=1$}%
    \addplot[blue,very thick,mark=triangle]
        table[x=deltat,y=LinftyL2u-all] {noGCC-restricted-1d-order2.dat};  \addlegendentry{$q=k=2$}%
    \addplot[green!70!black,very thick,mark=x]
        table[x=deltat,y=LinftyL2u-all] {noGCC-restricted-1d-order3.dat};  \addlegendentry{$q=k=3$}%

    \addplot[red,very thick,dashed,forget plot]
        table[x=deltat,y=LinftyL2u-restrict] {noGCC-restricted-1d-order1.dat}; 
    \addplot[blue,very thick,dashed,forget plot]
        table[x=deltat,y=LinftyL2u-restrict] {noGCC-restricted-1d-order2.dat}; 
    \addplot[green!70!black,very thick,dashed,forget plot]
	table[x=deltat,y=LinftyL2u-restrict] {noGCC-restricted-1d-order3.dat}; 

    \addplot[red,only marks,mark=*,forget plot]
        table[x=deltat,y=LinftyL2u-restrict] {noGCC-restricted-1d-order1.dat}; 
    \addplot[blue,only marks,mark=triangle,forget plot]
        table[x=deltat,y=LinftyL2u-restrict] {noGCC-restricted-1d-order2.dat}; 
    \addplot[green!70!black,only marks,mark=x,forget plot]
	table[x=deltat,y=LinftyL2u-restrict] {noGCC-restricted-1d-order3.dat}; 
    
    \addplot[lightgray,dashed,ultra thick,forget plot]
        table[mark=none,x=deltat,y expr ={0.07*\thisrowno{0}}] {noGCC-restricted-1d-order1.dat};  \addlegendentry{$ \mathcal{O}(\Delta t) $ } %
    \addplot[lightgray,dotted,ultra thick,forget plot]
        table[mark=none,x=deltat,y expr ={.04*\thisrowno{0}*\thisrowno{0} }] {noGCC-restricted-1d-order2.dat};  \addlegendentry{$ \mathcal{O}((\Delta t)^2) $ } %
    \addplot[lightgray,dashdotted,ultra thick,forget plot]
        table[mark=none,x=deltat,y expr ={.02*\thisrowno{0}*\thisrowno{0}*\thisrowno{0}}] {noGCC-restricted-1d-order3.dat};  \addlegendentry{$ \mathcal{O}((\Delta t)^3) $ } %

    \node[draw,circle,ultra thick,lightgray] (Z1) at (axis cs:0.03125,4.5e-6) {};
   \end{axis}
    \node at ($(displ) + (-0.0cm,-2.25cm)$) {\ref{grouplegendnoGCC}}; 
\end{tiny}
\end{scope}

 \begin{scope}[xshift=6.5cm]
\node (abserr) at (0,1) {\includegraphics[scale = 0.3 ]{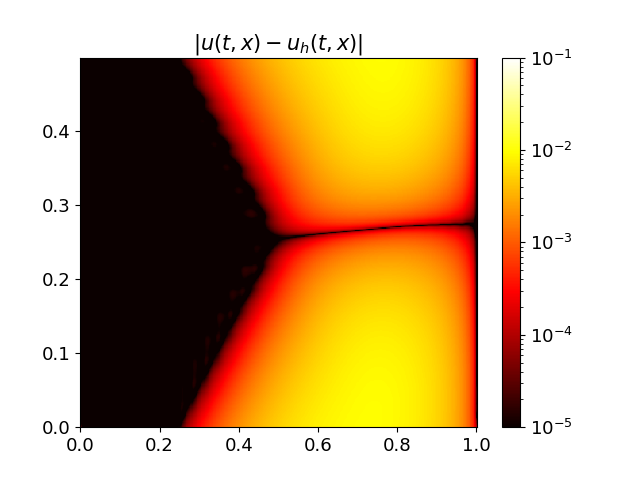}};
\draw[->,ultra thick] (-2.35,-0.95) -- node [right,at end] {$x$}  (1.0,-0.95);
\draw[->,ultra thick] (-2.35,-0.95) -- node [above,at end] {$t$}  (-2.35,2.5);
\draw[white,ultra thick,dashed] (-1.05,-0.5 ) -- (-0.3, 1.0);
\draw[white,ultra thick,dashed] (-0.3,1.0 ) -- (-1.05, 2.4);
\node[] (B) at (-1.3,0.95) {  $\textcolor{white}{B}$ };
\draw[lightgray,ultra thick, ->] (-2.0,1.15 ) -- (-4.8, 0.95);
 \end{scope}

\begin{scope}[xshift=8.65cm]
\begin{tiny}
\begin{axis}[ 
    height = 4.15cm,
    width = 5.45cm,
    name=vel,
    ymode=log,
    xmode=log,
    ymax = 1e-0,
    axis y line*=left,
    xlabel= { $\Delta t = h$},
    ytick = {1e-3,1e-5},
    legend style = { column sep = 10pt, legend columns = 1, legend to name = grouplegendnoGCCv,},
    x label style={at={(axis description cs:0.75,+0.31)},anchor=east},
    y tick label style={ xshift=2.85em,yshift=-0.4em },
    title = { $\partial_t u$ in $ \Omega$ (solid) vs $ B_t $ (dashed) },
    title style={at={(0.5,1.0735)},anchor=north},
    legend style={at={(0.5,-0.1)},anchor=north},
	]
    \addplot[red,very thick,mark=*,forget plot]
        table[x=deltat,y=L2L2ut-all] {noGCC-restricted-1d-order1.dat}; \addlegendentry{$q=k=1$}%
    \addplot[blue,very thick,mark=triangle,forget plot] 
        table[x=deltat,y=L2L2ut-all] {noGCC-restricted-1d-order2.dat};  \addlegendentry{$q=k=2$}%
    \addplot[green!70!black,very thick,mark=x,forget plot]
        table[x=deltat,y=L2L2ut-all] {noGCC-restricted-1d-order3.dat};  \addlegendentry{$q=k=3$}%

    \addplot[red,very thick,mark=*,dashed,forget plot]
        table[x=deltat,y=L2L2ut-restrict] {noGCC-restricted-1d-order1.dat};  
    \addplot[blue,very thick,mark=triangle,dashed,forget plot]
        table[x=deltat,y=L2L2ut-restrict] {noGCC-restricted-1d-order2.dat}; 
    \addplot[green!70!black,very thick,mark=x,dashed,forget plot]
        table[x=deltat,y=L2L2ut-restrict] {noGCC-restricted-1d-order3.dat};  

    \addplot[lightgray,dashed,ultra thick]
        table[mark=none,x=deltat,y expr ={1.3*\thisrowno{0}}] {noGCC-restricted-1d-order1.dat};  \addlegendentry{$ \mathcal{O}(\Delta t) $ } %
    \addplot[lightgray,dotted,ultra thick]
        table[mark=none,x=deltat,y expr ={0.4*\thisrowno{0}*\thisrowno{0} }] {noGCC-restricted-1d-order2.dat};  \addlegendentry{$ \mathcal{O}((\Delta t)^2) $ } %
    \addplot[lightgray,dashdotted,ultra thick]
        table[mark=none,x=deltat,y expr ={.1*\thisrowno{0}*\thisrowno{0}*\thisrowno{0}}] {noGCC-restricted-1d-order3.dat};  \addlegendentry{$ \mathcal{O}((\Delta t)^3) $ } %
   \end{axis}
    \node at ($(vel) + (-0.0cm,-2.25cm)$) {\ref{grouplegendnoGCCv}}; 
\end{tiny}
\end{scope}
\end{tikzpicture}
\end{center}
\caption{ Results for $\Omega = [0,1]$ with data given in $[0,1/4]$ computed with the `M-f' method. 
On the left: $\norm{ u - \mathcal{L}_{\Delta t} \underline{u}_1 }_{L^{\infty}(0,T;L^2(\Omega))}$ (solid lines) and $\norm{ u - \mathcal{L}_{\Delta t} \underline{u}_1 }_{L^{\infty}(0,T;L^2(B_t))}$ (dashed lines) for the restricted set $B_t$ defined in \cref{def:Bt}. Similarly on the right for the time derivative in the $L^2-L^2$-norm. }
\label{fig:NoGCC-1D-B}
\end{figure}
To further investigate this phenomenon, let us consider an analogous experiment in one spatial dimension. Here, we consider $\Omega = [0,1]$ with data domain $\omega = [0,1/4]$, final time $T=1/2$ and exact solution $u(t,x) = \cos(\pi t) \sin(\pi x)$. The results shown in \cref{fig:NoGCC-1D-B} are similar to the spatially three-dimensional case. If we measure in the full spatial domain $\Omega$ (solid lines in the plot), then the error appears to stagnate. The method is unable to reconstruct the solution far away from the data domain as the space-time plot of the absolute error in the center of \cref{fig:NoGCC-1D-B} demonstrates. However, if we settle for measuring the spatial error in the sets $B_t$ defined as 
\begin{equation}\label{def:Bt}
B = \cup_{t \in [0,T]} B_t, \quad B_t := \{ x \leq 1/4 + t \text{ for }  t \leq 1/4 \mid x \leq 3/4 - t \text{ for }  t \geq 1/4\},
\end{equation}
then optimal convergence can be observed (see the dashed lines in \cref{fig:NoGCC-1D-B}). It would hence be interesting to investigate whether the Lipschitz stability result from \cref{thm:Lipschitz-stab} could be adapted to cover the considered setting. 

\appendix

\section{Interpolation in space-time}\label{appendix:space-time-interp}

In reference \cite{preussmaster} an interpolation theory for unfitted space-time finite element spaces has been developed. 
Here we recall the results for the fitted case and give some extensions which are 
required for the applications in this paper. The final aim of this appendix is to prove \cref{lem:interp-slabwise} and \cref{lem:approx-fixed-time-lvl}.  

We first define interpolation on the extended time slices 
$\hat{Q}^n = I_n \times \hat{\Omega}$ and $\hat{Q} = I \times \hat{\Omega}$
and then obtain the result on the orignal domains by means of a Sobolev extension argument. 
To this end, let 
\[
\hat{V}_h^k := \{ u \in H^1(\hat{\Omega}) \mid u|_K \in \mathcal{P}_k(K) \; \forall K \in \hat{\mathcal{T}}_h \}, \qquad \FullyDiscrSpaceHat{k}{q} :=  \otimes_{n=0}^{N-1} \mathcal{P}^q(I_n) \otimes \hat{V}_h^k.
\]
We start in \cref{ssection:interp-time} by defining a purely temporal interpolation operator and introduce its purely spatial counterpart in \cref{ssection:interp-in-space}.
In \cref{ssection:space-time-interp} we then concatenate these operators to achieve space-time interpolation. To seprate spatial and temporal continuity, we define $t$-anisotropic Sobolev spaces 
\bel{eq:t-anisotropic-Sobolev-spaces}
H^{q,k}(\hat{Q}) := \left\{ u \mid \partial_t^p D_x^{\alpha}u \in L^2(\hat{Q}), \frac{\abs{\alpha}}{k} + \frac{p}{q} \leq 1  \right\}.   
\ee
For $q=k$ the isotropic Sobolev spaces are retained: $H^{k,k}(\hat{Q}) = H^{k}(\hat{Q})$.  

\subsection{Interpolation in time}\label{ssection:interp-time}
For $q \in \mathbb{N}_0$ define $\pi^{q}_n: H^{0,0}(\hat{Q}^n) \rightarrow \mathcal{P}^q(I_n) \otimes L^2(\hat{\Omega})  $ by 
\begin{equation}\label{eq:Def temporal L2-projection}
(  \pi^{q}_n(u) - u,v)_{\hat{Q}^{n}} = 0 \quad \forall v \in  \mathcal{P}^{q}(I_n) \otimes L^2(\hat{\Omega}).
\end{equation}
On sufficiently smooth functions, $\pi^{q}_n$ coincides with a standard temporal $L^2(I_{n})$-projection $I^{q}_{n}$ defined by 
\begin{equation*}
	(I^{q}_{n} u,\chi)_{I_{n}} = (u,\chi)_{I_{n}} \quad \forall \; \chi \in \mathcal{P}^{q}(I_n).
\end{equation*}

\begin{lemma}\label{Lemma:Temporal interpolation commute with spatial derivatives}
For $u \in C^{0}( \hat{\Omega} ;L^2(I_{n}))$ there holds 
\begin{equation*}
\pi^{q}_n u = I_n^{q}u \text{ in the } L^2(\hat{Q}^{n}) \text{ sense.}
\end{equation*}
Moreover, for $u \in H^{0,1}(\hat{Q}^{n})$ the temporal interpolation commutes with spatial derivatives 
\begin{equation*}
\partial_{x_{j}} \pi^{q}_{n} u = \pi^{q}_{n} \partial_{x_{j}}u, \quad j=1,\ldots,d.
\end{equation*}
\end{lemma}
\begin{proof}
See \cite[Lemma 3.14]{preussmaster}. 
\end{proof}
We denote the temporal interpolation error by $\mathcal{L}^{q}_n := \mathrm{id} - \pi^{q}_n $. 
The stability and approximation properties of  $ \pi^{q}_{n}$ are studied below.
\begin{lemma}\label{Lemma:Stability and approximation L^2 projection in time}
The following stability and approximation results hold. 
\begin{enumerate}[label=\emph{(\alph*)}]
\item $\norm{  \pi^{q}_{n} u }_{\hat{Q}^{n}} \leq \norm{u}_{\hat{Q}^{n}} \quad \forall u \in L^2(\hat{Q}^{n})$ for $n=1,\ldots,N-1$.
\item $\norm{ \partial_{t} \pi^{q}_{n}  u}_{\hat{Q}^{n}} \leq C \norm{u}_{H^{1,0}(\hat{Q}^{n})} \quad \forall u \in H^{1,0}(\hat{Q}^{n})$ for $n=1,\ldots,N-1$.
\item For $n=1,\ldots,N-1 $ we have $\norm{\mathcal{L}^{q}_{n} u}_{ \hat{Q}^{n}} \leq C \Delta{t}^{l} \norm{u}_{H^{l,0}(\hat{Q}^{n})} $ $\forall u \in H^{l,0}(\hat{Q}^{n})$ and $ l \in \{0,\ldots,q+1 \}$.
\item For $n=1,\ldots,N-1 $ we have $\norm{ \partial_{t} \mathcal{L}^{q}_{n} u}_{ \hat{Q}^{n}} \leq C \Delta{t}^{l-1} \norm{u}_{H^{l,0}( \hat{Q}^{n})}$ for all $ u \in H^{l,0}(\hat{Q}^{n})$ and $ l \in \{1,\ldots,q+1 \}$.
\end{enumerate}
\end{lemma}
\begin{proof}
The results for $ \Pi^{q}_{n}$ have been shown in \cite[Lemma 3.15]{preussmaster}. 
Note that in this reference the results (c) and (d) are only stated for $ l \in \{1,\ldots,q+1 \}$ and  $ l \in \{2,\ldots,q+1 \}$, respectively.
However, the case $l=0$ in (c) immediately follows from (a) and the case $l=1$ in (d) follows from (b). \par 
\end{proof}

\subsection{Interpolation in space}\label{ssection:interp-in-space}

In this section we consider a generic interval $\tilde{I} \subset (0,T)$ and set $ \tilde{Q} = \tilde{I} \times \hat{\Omega}$. Later, we will choose $\tilde{I} = (0,T)$ to define an interpolation operator into the semi-discrete space $\SemiDiscrSpace{k}$ and $\tilde{I} = I_n$ to define a time-slab-wise interpolation into the fully-discrete space $\FullyDiscrSpace{k}{q}$. 

Define the projector $ \pi^{k} : L^2(\tilde{Q}) \rightarrow L^2(\tilde{I}) \otimes  \hat{V}_{h}^k $, so that for $u \in L^2(\tilde{Q})$
\begin{equation}
(\pi^{k}(u) -u,v_{h})_{\tilde{Q}} = 0, \quad \text{ for all } v_{h} \in L^2(\tilde{I}) \otimes  \hat{V}_{h}^k. 
\end{equation}
On sufficiently smooth functions $\pi^{k}$ coincides with a spatial $L^2( \hat{\Omega} )$ projector $I^{k}: L^2( \hat{\Omega} ) \rightarrow \hat{V}_{h}^k$.
\begin{lemma}\label{Lemma:Spatial interpolation commute with temporal derivatives}
Let $I^{k}: L^2( \hat{\Omega})  \rightarrow \hat{V}_{h}^k$ be the spatial $L^2( \hat{\Omega} )$ projection operator. Then there holds for $u \in C^{0}( \tilde{I},L^2( \hat{\Omega} ))$
\begin{equation}
\pi^{k}u = I^{k}u \quad \text{in the } L^2( \tilde{Q} ) \text{ sense.}
\end{equation}
Further, for $u \in H^{1,0}(\tilde{Q})$ there holds: 
\begin{equation}
\pi^{k} \partial_{t} u = \partial_{t} \pi^{k} u.
\end{equation}
\end{lemma}
\begin{proof}
See \cite[Lemma 3.3.5]{L_PHD_2015}.
\end{proof}
Let $ \mathcal{L}^{k} = \mathrm{id} - \pi^{k}$ denote the spatial interpolation error. The next lemma shows that the projection into the semi-discrete space is stable and has the expected approximation properties. 

\begin{lemma}\label{Lemma:Stability and error estimates spatial interpolation}
For the projector $\pi^{k}: L^2(\tilde{Q}) \rightarrow L^2(\tilde{I}) \otimes  \hat{V}_{h}^k $ there hold the following stability and approximation results: 
\begin{enumerate}[label=\emph{(\alph*)}]
\item $\norm{\pi^{k}u}_{\tilde{Q}} \leq \norm{u}_{\tilde{Q}} \quad \forall u \in L^2(\tilde{Q})$,
\item $\norm{ \nabla \pi^{k} u}_{ \tilde{Q} } \leq C \norm{ u}_{H^{0,1}(\tilde{Q})} \quad \forall u \in H^{0,1}(\tilde{Q})$,
\item $\norm{\mathcal{L}^{k} u}_{\tilde{Q}} \leq C h^{s} \norm{u}_{H^{0,s}(\tilde{Q})} \quad \forall u \in H^{0,s}(\tilde{Q}), \quad s \in \{1,\ldots,k+1 \}$,
\item $\norm{ \nabla \mathcal{L}^{k} u}_{\tilde{Q}} \leq C h^{s-1} \norm{u}_{H^{s,0}(\tilde{Q})} \quad \forall u \in H^{0,s}(\tilde{Q}), \quad s \in \{1,\ldots,k+1 \}$.
\item Also, for all $ \forall u \in H^{0,s}(\tilde{Q}) \cap C^{0}(\tilde{I},H^2(\hat{\Omega} ))$ and $ s \in \{2,\ldots,k+1 \}$ it holds that
\[ \left( \int\limits_{\tilde{I}} \sum\limits_{K \in \hat{\mathcal{T}}_h } \norm{\mathcal{L}^{k} u}_{ H^2(K) }^2 \; \dT \right)^{1/2} \leq C h^{s-2} \norm{u}_{H^{0,s}(\tilde{Q})}. \]  
\end{enumerate}
\end{lemma}
\begin{proof}
We only show (e) since the results (a)-(d) have already been established in \cite[Lemma 3.17]{preussmaster}.
Note that since $u \in C^{0}(\tilde{I},L^2( \hat{\Omega} ))$ we have $\mathcal{L}^{k}u = u - I^{k} u$. 
Moreover, as $u(t,\cdot) \in H^2(\hat{\Omega} )$ and $d \leq 3$ the standard Lagrange interpolation operator $I^{L}:H^2( \hat{\Omega}  ) \rightarrow \hat{V}_{h}^k $ is well-defined 
and has the approximation property
\bel{eq:Langrange-interp-H2}
\norm{I^{L}u - u}_{H^2(K)} \leq C h^{s-2} \norm{u}_{H^s(K)}, \; \text{for } s \in \{ 2, \ldots, k+1 \}.  
\ee
Then we have 
\[ \norm{\mathcal{L}^{k} u}_{ H^2(K) } \leq \norm{ u -  I^{L}u }_{H^2(K)} + \norm{ I^{L}u - I^{k} u }_{H^2(K)}.   \]
The integrand in the second term is a discrete function to which the usual inverse inequalities can be applied: 
\begin{align*}
\norm{ I^{L}u - I^{k} u }_{H^2(K)} &\leq  h^{-2} C \norm{ I^{L}u - I^{k} u }_{L^2(K)}  \\
& \leq h^{-2} C \left(  \norm{ I^{L}u - u }_{L^2(K)} + \norm{ I^{k} u - u }_{L^2(K)}   \right). 
\end{align*}
Note that $\norm{ I^{k} u - u }_{L^2( \hat{\Omega}  ) } \leq  \norm{ I^{L}u - u }_{L^2( \hat{\Omega}  )}$ by the best approximation property of the 
$L^2(\hat{\Omega} )$-projection.
	Therefore, summing up over $K \in \hat{\mathcal{T}}_h$ yields
\[
h^{-2} \sum\limits_{K \in \hat{\mathcal{T}}_h }  \norm{ I^{k} u - u }_{L^2(K)}^2 = h^{-2} \norm{ I^{k} u - u }_{L^2(\hat{\Omega} )}^2 \leq h^{-2} \norm{ I^{L} u - u }_{L^2(\hat{\Omega} )}^2.
\]
Combining these results it follows that
\begin{align*}
\int\limits_{\tilde{I}} \sum\limits_{K \in \hat{\mathcal{T}}_h } \norm{\mathcal{L}^{k} u}_{ H^2(K) }^2 \; \dT 
&\leq C \int\limits_{ \tilde{I}}  \sum\limits_{K \in \hat{\mathcal{T}}_h  } \left\{  \norm{ I^{L}u - u }_{H^2(K)}^2 + h^{-2} \norm{ I^{L}u - u }_{L^2(K)}^2  \right\} \dT   \\  
& \leq C h^{2(s-2)} \norm{u}_{H^{0,s}(\tilde{Q} )}^2, \; s \in \{2,\ldots,k+1 \}.
\end{align*}
\end{proof}

Using the result established above it is easy to deduce the interpolation results 
required for analysis of the semi-discretization. \par 
\noindent \textbf{Proof of \cref{lem:interp-slabwise} for $ \Pi_h = \Pi_h^k$:}  \par 
\noindent To define an interpolation operator into $\SemiDiscrSpace{k}$, we precompose $\pi^{k}$ defined on $\tilde{I} = (0,T)$ with a continuous (and degree independent) Sobolev extension operator $\mathcal{E}: H^s(\STdom) \rightarrow H^s(\hat{Q})$, see \cite[Theorem 5, Chapter VI]{S70}. That is, let $\Pi_h^k u :=  (\pi^{k} \mathcal{E}u)|_{Q}$. Then by construction  $\Pi_h^k u \in  L^2(\tilde{I}) \otimes  V_{h}^k $. Thanks to \cref{Lemma:Spatial interpolation commute with temporal derivatives} we have for sufficiently regular $u$ that
\[
\partial_t \Pi_h^k u =   (\partial_t \pi^{k} \mathcal{E}u)|_{Q}.
= ( \pi^{k} \partial_t \mathcal{E}u)|_{Q} \in  L^2(\tilde{I}) \otimes  V_{h}^k. 
\]
Hence, $\Pi_h^k u  \in H^1 \left( 0,T; V_h^k \right) =  \SemiDiscrSpace{k}$. 
To apply the interpolation results established in \cref{Lemma:Stability and error estimates spatial interpolation} to $\Pi_h^k$, we note that for any $D \in \{ \mathrm{id},\nabla,\partial_t \}$ 
we have 
\begin{align*}
\norm{D ( \Pi_h^ku -u ) }_{Q} = \norm{D ( \pi^{k} \mathcal{E}u  - \mathcal{E} u ) }_{Q} 
\leq  \norm{D ( \pi^{k} \mathcal{E}u  - \mathcal{E} u ) }_{ \hat{Q}}. 
\end{align*}
Using the interpolation results from \cref{Lemma:Stability and error estimates spatial interpolation} then allows to estimate the last expression by some power of $h$ 
times $C \norm{  \mathcal{E}u }_{ H^s( \hat{Q}) }$ which can be bounded by $C \norm{u }_{ H^s( Q) }$ thanks to continuity of the Sobolev extension operator. Using this argument the interpolation results in \cref{lem:interp-slabwise} follow from \cref{Lemma:Stability and error estimates spatial interpolation}.   

\subsection{Interpolation in space-time}\label{ssection:space-time-interp}
To define an interpolation operator into the fully-discrete spaces $\FullyDiscrSpace{k}{q}$, some additional work is necessary. If we apply the results from \cref{ssection:interp-in-space} on $\tilde{I} = I_n$ we obtain an interpolation operator $\pi^{k}_n$ into $L^2(I_n) \otimes  \hat{V}_{h}^k$ for $n=0,\ldots,N-1$. The space-time interpolation operator $\pi^{q,k}_n$ is then obtained by concatenating the purely spatial and temporal operators introduced in the previous two subsections. We define 
\begin{equation}\label{eq:interp_time_slab}
\pi_{n}^{q,k}: H^{0,0}(\hat{Q}^n) \rightarrow \mathcal{P}^q(I_n) \otimes \hat{V}_h^k, \quad \pi^{q,k}_{n} = \pi^{q}_{n} \circ \pi^{k}_n 
\end{equation}
and obtain the global operator by means of restriction to the time-slabs: 
\begin{equation}\label{eq:Definition space-time interpolation operator_DG}
\Pi^{q,k}_{n}: H^{0,0}(\hat{Q}) \rightarrow \FullyDiscrSpaceHat{q}{k}, \quad \Pi^{q,k}_{n}u =  \pi^{q,k}_{n} u|_{ \hat{Q}^n }.  
\end{equation}
An important ingredient for establishing the approximation results of these operators is that spatial derivatives commute with the temporal interpolation and vice versa (cf. \cref{Lemma:Temporal interpolation commute with spatial derivatives} and \cref{Lemma:Spatial interpolation commute with temporal derivatives}).

\begin{theorem}\label{Thm:Erorror bounds for space-time interpolation on time slab}
For $l_{q},l_{k} \in \mathbb{N}$ define $l_{\max} = \max{\{l_{q},l_{k} \}}$.
The following stability results and approximation error bounds hold:
\begin{enumerate}[label=\emph{(\alph*)}]
\item $ \norm{ \pi^{q,k}_{n}  u}_{L^2(\hat{Q}^{n})} \leq \norm{u}_{L^2(\hat{Q}^{n})} $ and $ \norm{ \pi^{q,k}_{n}  u}_{H^1(\hat{Q}^{n})} \leq C  \norm{u}_{H^1(\hat{Q}^{n})}  $.
\item $\norm{u -  \pi^{q,k}_{n} u}_{\hat{Q}^{n}} \leq C \left( \Delta{t}^{l_{q}}  + h^{l_{k}}  \right) \norm{u}_{H^{l_{\max} }(\hat{Q}^{n})}$, 
for $l_{k} \in \{1,\ldots,k + 1 \}$ and $l_{q} \in \{1,\ldots,q+1 \}$.
\item $\norm{\partial_{t} (u - \pi^{q,k}_{n} u) }_{\hat{Q}^{n}} \leq C \left( \Delta{t}^{l_{q}-1}  + h^{l_{k}}\right) \norm{u}_{H^{l_{\max}+1 }(\hat{Q}^{n})}$, for $l_{k} \in \{1,\ldots,k +1 \}$ and $l_{q} \in \{2,\ldots,q+1 \}$. 
\item $\norm{\nabla{ (u -  \pi^{q,k}_{n} u) }}_{\hat{Q}^{n}} \leq C \left( \Delta{t}^{l_{q}}  + h^{l_{k}-1}  \right) \norm{u}_{H^{l_{\max}+1 }(\hat{Q}^{n})}$, for $l_{k} \in \{2,\ldots,k +1 \}$ and $l_{q} \in \{1,\ldots,q+1 \}$. 
\item For $u \in H^{l_{\max}+2 }(\hat{Q}^{n}) \cap C^{0}(I_{n},H^2(\hat{\Omega} ))$ we have 
\[  \left( \int\limits_{I_n} \sum\limits_{K \in \hat{\mathcal{T}}_h } \norm{ u - \pi^{q,k}_{n} u }_{ H^2(K) }^2 \; \dT \right)^{1/2} \leq C \left( \Delta{t}^{l_{q}}  + h^{l_{k}-2}  \right) \norm{u}_{H^{l_{\max}+2 }(\hat{Q}^{n})}  \]
for $l_{k} \in \{2,\ldots,k +1 \}$ and $l_{q} \in \{1,\ldots,q+1 \}$.
\end{enumerate}
\end{theorem}
\begin{proof}
We only show (a) and (e) as the other results have been established in \cite[Corollary 3.19]{preussmaster}.
\begin{enumerate}
\item[(a)] 
From \cref{Lemma:Stability and approximation L^2 projection in time} (a) and \cref{Lemma:Stability and error estimates spatial interpolation} (a) we obtain
\[ \norm{ \pi^{q,k}_{n}  u}_{\hat{Q}^{n}} = \norm{  \pi^{q}_{n}  \pi^{k}_n  u}_{\hat{Q}^{n}} \leq \norm{  \pi^{k}_n u }_{\hat{Q}^{n}} \leq  \norm{ u }_{\hat{Q}^{n}}.   \]
By using that spatial derivatives commute with the interpolation in time (see \cref{Lemma:Temporal interpolation commute with spatial derivatives}) and applying 
the stability results of \cref{Lemma:Stability and approximation L^2 projection in time} (a) and \cref{Lemma:Stability and error estimates spatial interpolation} (b)
we get
\[
\norm{\nabla \pi^{q,k}_{n}   u}_{\hat{Q}^{n}} =  \norm{  \pi^{q}_{n}  \nabla \pi^{k}_n u}_{\hat{Q}^{n}} \leq \norm{ \nabla \pi^{k}_n u}_{\hat{Q}^{n}} \leq C \norm{ u}_{H^{0,1}(\hat{Q}^{n})}. 
\]
From the stability of the temporal interpolation, \cref{Lemma:Stability and approximation L^2 projection in time} (b), we have 
		\[  \norm{\partial_t \pi^{q,k}_{n} u}_{Q^{n}} = \norm{ \partial_t \pi^{q}_{n}  ( \pi^{k}_n u) }_{\hat{Q}^{n}} \leq C  \norm{  \pi^{k}_n u }_{ H^{1,0}(\hat{Q}^{n}) }.   \]
Since temporal derivatives commute with the spatial interpolation (as shown in \cref{Lemma:Spatial interpolation commute with temporal derivatives}), we have 
\[
	\norm{  \pi^{k}_n u }_{ H^{1,0}(\hat{Q}^{n}) }^2 = \norm{ \pi^{k}_n \partial_t u }_{\hat{Q}_n}^2 + \norm{ \pi^{k}_n u }_{\hat{Q}_n}^2 \leq C  \norm{ u}_{H^{1,0}(\hat{Q}^{n})}^2
\]
by means of  \cref{Lemma:Stability and error estimates spatial interpolation} (a).
\item[(e)]
We utilize the splitting $u- \pi^{q,k}_{n}  u = \mathcal{L}^q u + \pi^{q}_{n} \mathcal{L}^{k} u$. 
Hence, 
\begin{equation*}
\int\limits_{I_n} \sum\limits_{K \in \hat{\mathcal{T}}_h } \norm{ u - \pi^{q,k}_{n} u }_{ H^2(K) }^2 
\leq C  \int\limits_{I_n} \sum\limits_{K \in \mathcal{T}_h^n } \left\{  \norm{ \mathcal{L}^q u }_{ H^2(K) }^2 + \norm{ \pi^{q} \mathcal{L}^{k} u }_{ H^2(K) }^2   \right\}.  
\end{equation*}
Since spatial derivatives commute with $ \mathcal{L}^q $ (see \cref{Lemma:Temporal interpolation commute with spatial derivatives}), we obtain from \cref{Lemma:Stability and approximation L^2 projection in time} (c): 
\[
	\int\limits_{I_n} \sum\limits_{K \in \mathcal{T}_h }  \norm{ \mathcal{L}^q u }_{ H^2(K) }^2 \; \dT \leq C \Delta{t}^{2 l_q} \norm{u}_{H^{l_q,2}(\hat{Q}^{n})}^2 \quad  l_q \in \{0,\ldots,q+1 \}.
\]
For the other term we use stability of the temporal interpolation (see the \cref{Lemma:Stability and approximation L^2 projection in time} (a)) and 
\cref{Lemma:Stability and error estimates spatial interpolation} (e) to obtain 
\begin{align*}
\int\limits_{I_n} \sum\limits_{K \in \hat{\mathcal{T}}_h } \norm{ \pi^{q}_n  \mathcal{L}^{k} u }_{ H^2(K) }^2 \; \dT 
& \leq 
\int\limits_{I_n} \sum\limits_{K \in \hat{\mathcal{T}}_h } \norm{  \mathcal{L}^{k}u }_{ H^2(K) }^2 \; \dT \\
	& \leq 
	C h^{2(l_k-2)} \norm{u}_{H^{0,l_k}(\hat{Q}^{n})}^2
\end{align*}
for $l_k \in \{2,\ldots,k+1 \}$ which yields the desired result. \par 
\end{enumerate}
\end{proof}
To derive a bound on the approximation error at discrete time levels from the previous results, we need an additional discrete inverse inequality in time:
\begin{lemma}[Discrete inverse inequality in time]\label{Lemma:Discrete inverse inequality in time}
Let $M$ be a bounded Lipschitz domain.
There exists a constant $C > 0$ such that
\begin{equation}\label{eq:Discrete inverse inequality in time}
\norm{u(\tau,\cdot) }_{M} \leq C \Delta{t}^{-1/2} \norm{u}_{I_n \times M }
\end{equation}
for all $u \in  \mathcal{P}^{q}(I_n) \otimes L^2(M)$ and $\tau \in \{ t_n,t_{n+1}\}$.
\end{lemma}
\begin{proof}
See \cite[Lemma 3.20]{preussmaster}. In this reference the case $\tau = t_n$ has been shown, however, the proof for $\tau = t_{n+1}$ is completely analogous.
\end{proof}

\begin{lemma}[Approximation error bounds at fixed time levels]\label{Lemma: space-time interpolation error estimate at fixed time level} 
At the top and bottom of the time slice the following approximation bounds for the traces hold true:
\begin{enumerate}[label=(\alph*)]
\item For $u \in H^{l_{q},0}(\hat{Q}^{n}) \cap H^{0,l_{k}}(\hat{Q}^{n})$ we have  
\begin{equation}\label{eq: space-time interpolation error estimate at fixed time level}
\norm{(u - \pi^{q,k}_{n} u )( \tau, \cdot )}_{\hat{\Omega } } \leq C \left( \Delta{t}^{l_{q}-1/2} \norm{u}_{H^{l_{q},0}(\hat{Q}^{n})} + \Delta{t}^{-1/2} h^{l_{k}} \norm{u}_{H^{0,l_{k}}(\hat{Q}^{n})}  \right), 
\end{equation}
with $\tau \in \{ t_n,t_{n+1}\}$ for $l_{k} \in \{1,\ldots,k_{s} +1 \}$ and $l_{q} \in \{1,\ldots,q+1 \}$, 
\item For $u \in H^{l_{q},0}(\hat{Q}^{n}) \cap H^{0,l_{k}}(\hat{Q}^{n})$ we have 
\begin{align}
& \norm{\nabla (u - \pi^{q,k}_{n} u )( \tau, \cdot )}_{\hat{\Omega} } \label{eq:time_lvl_approx_grad}
\\
& \leq C \left( \Delta{t}^{l_{q}-1/2} \norm{\nabla u}_{H^{l_{q},0}(\hat{Q}^{n})} + \Delta{t}^{-1/2} h^{l_{k}-1} \norm{u}_{H^{0,l_{k}}(\hat{Q}^{n})}  \right), \nonumber
\end{align}
with $\tau \in \{ t_n,t_{n+1}\}$, for $l_{k} \in \{1,\ldots,k_{s} +1 \}$ and $l_{q} \in \{1,\ldots,q+1 \}$.
\end{enumerate}
\end{lemma}
\begin{proof}
We will show (b). The proof of (a) is analogous and can be found in \cite[Lemma 3.21]{preussmaster}. 
Recall the splitting $u- \pi^{q,k}_{n}  u = \mathcal{L}^q u + \pi^{q}_{n} \mathcal{L}^{k} u$. 
An application of the triangle inequality yields 
\[	
\norm{\nabla (u - \pi^{q,k}_{n}  u )( \tau, \cdot )}_{ \hat{\Omega} } 
\leq \norm{ \nabla ( \mathcal{L}^q u  )( \tau, \cdot )}_{ \hat{\Omega} } 
+ \norm{ \nabla ( \pi^{q}_{n} \mathcal{L}^{k} u )( \tau, \cdot )}_{ \hat{\Omega} } 
\]
To treat the second term, we can apply the discrete inverse inequality in time, see the \cref{Lemma:Discrete inverse inequality in time}, which yields  
\begin{align*}
& \norm{ \nabla ( \Pi^{q}_{n} \mathcal{L}^{k} u )( \tau, \cdot )}_{ \hat{\Omega}  } 
= \norm{ \Pi^{q}_{n} \nabla \mathcal{L}^{k} u ( \tau, \cdot )}_{ \hat{\Omega} } \\ 
& \leq C (\Delta t)^{-1/2} \norm{ \nabla \mathcal{L}^{k} u }_{ \hat{Q}^n } 
 \leq C (\Delta t)^{-1/2} h^{ l_{k} -1 } \norm{u}_{H^{0,l_{k}}(\hat{Q}^{n})}  
\end{align*}
for $l_{k} \in \{1,\ldots,k_{s} +1 \}$,  where the spatial approximation results from \cref{Lemma:Stability and error estimates spatial interpolation} have been employed in the last step.
	To treat the term $  \norm{ \nabla ( \mathcal{L}^q u  )( \tau, \cdot )}_{\hat{\Omega}} = \norm{  \mathcal{L}^q g  ( \tau, \cdot )}_{\hat{\Omega}}  $ with $g=\nabla u$, 
we transform the interval $I_n$ to the reference interval $\tilde{I} = (0,1]$ and consider the problem there. 
Denote the corresponding transformation, which only depends on the temporal variable, by $\Phi: \tilde{Q} \rightarrow \hat{Q}^{n}$ with $\tilde{Q} = \tilde{I} \times \hat{\Omega} $. 
The transformed function and operator are given by $\tilde{g} = g \circ \Phi$ and $ \tilde{\mathcal{L}}^q = \Phi^{-1} \circ \mathcal{L}^q \circ \Phi $.   
Using continuity of the time trace operator in $H^{1,0}(\tilde{Q})$ and the interpolation estimates from \cref{Lemma:Stability and approximation L^2 projection in time} (c) for $ \mathcal{L}^q$ yields (denoting $\tilde{\tau} = \Phi^{-1}(\tau)$)
\begin{align*}
& \norm{  \mathcal{L}^q g  ( \tau, \cdot )}_{\hat{\Omega}  }^2 
=  \norm{ \tilde{\mathcal{L}}^q  \tilde{g}  ( \tilde{\tau}, \cdot )}_{\hat{\Omega} }^2 
\leq C \norm{ \tilde{\mathcal{L}}^q  \tilde{g} }_{ H^{1,0}(\tilde{Q}) }^2 
= C \left(  \norm{ \tilde{\mathcal{L}}^q  \tilde{g} }_{ \tilde{Q} }^2 + \norm{ \partial_t \tilde{\mathcal{L}}^q  \tilde{g} }_{ \tilde{Q} }^2   \right) \\ 
& \leq C  \left( \frac{1}{ \Delta t }   \norm{ \mathcal{L}^q  g }_{ \hat{Q}^{n} }^2 + \Delta t \norm{ \partial_t \mathcal{L}^q  g }_{ \hat{Q}^{n} }^2   \right) \\
& \leq C \left( \frac{ (\Delta t)^{2  l_{q} } }{ \Delta t } \norm{g}_{ H^{l_{t},0}(\hat{Q}^n) }^2 + \Delta (\Delta t)^{ 2(l_{q} -1  )  } \norm{g}_{ H^{l_{t},0}(\hat{Q}^n) }^2  \right) 
\end{align*}
for $l_{q} \in \{1,\ldots,q+1 \}$. 
Combining these estimates yields the claim.
\end{proof}
To transfer the results of \cref{Thm:Erorror bounds for space-time interpolation on time slab} and \cref{Lemma: space-time interpolation error estimate at fixed time level} from $\hat{Q}^{n}$ to $Q^{n}$ we proceed analogous to the semi-discrete case. \par  
\noindent \textbf{Proof of \cref{lem:interp-slabwise} and \cref{lem:approx-fixed-time-lvl} for $ \Pi_h = \Pi_{h, \Delta t}^{k,q} $} \par 
\noindent As before let $\mathcal{E}: H^s(\STdom) \rightarrow H^s(\hat{Q})$ be a continuous Sobolev etension operator. We define $\Pi_{h, \Delta t}^{k,q}$ time-slab wise by 
\[
\Pi_{h, \Delta t}^{k,q}u|_{Q^n} =   (\pi^{q,k}_{n} \mathcal{E}u)|_{Q^n} \in \mathcal{P}^q(I_n) \otimes V_h^k.
\]
Then for $D \in \{ \mathrm{id}, \partial_t, \nabla, \partial_{x_i} \partial_{x_j} \}$ we have 
\begin{align*}
\norm{D ( \Pi_{h, \Delta t}^{k,q} u -u ) }_{Q^n} = \norm{D ( \pi^{q,k}_{n} \mathcal{E}u  - \mathcal{E} u ) }_{Q^n} 
\leq  \norm{D ( \pi^{q,k}_{n} \mathcal{E}u  - \mathcal{E} u ) }_{ \hat{Q}^n}. 
\end{align*}
Then using \cref{Thm:Erorror bounds for space-time interpolation on time slab} and continuity of the Sobolev extension operator this can be bounded by an appropriate power of $h$ (specified in the theorem) times $C \norm{u}_{H^s(Q^n)}$.
We proceed similarly for \cref{lem:approx-fixed-time-lvl}. With this argument we obtain \cref{lem:interp-slabwise} from \cref{Thm:Erorror bounds for space-time interpolation on time slab} and \cref{lem:approx-fixed-time-lvl} from \cref{Lemma: space-time interpolation error estimate at fixed time level}.


\bibliographystyle{siamplain}
\bibliography{references}
\end{document}